\documentclass[12pt]{article}
\usepackage[utf8]{inputenc}
\usepackage{graphicx}
\usepackage[a4paper, margin=1.2in]{geometry}  
\usepackage{amsmath, amsfonts}
\usepackage{tcolorbox}
\usepackage{amssymb}
\usepackage{amsthm}
\usepackage{lastpage}
\usepackage{fancyhdr}
\usepackage{accents}
\usepackage{enumitem}
\usepackage{tikz-cd}
\usepackage{xcolor}
\usepackage[british]{babel}
\usepackage[backend=biber,maxnames=99]{biblatex}
\usepackage{tikzit}
\usepackage{subcaption}
\usepackage{caption}
\usepackage{float}

\tikzstyle{small_black_dot}=[fill=black, draw=black, shape=circle]
\tikzstyle{what is this}=[fill=white, draw=black, shape=circle]
\tikzstyle{lookingempty}=[fill=white, draw=none, shape=circle]
\tikzstyle{purpledot}=[fill={rgb,255: red,128; green,0; blue,128}, draw={rgb,255: red,128; green,0; blue,128}, shape=circle]
\tikzstyle{bluedot}=[fill=blue, draw=blue, shape=circle]
\tikzstyle{reddot}=[fill=red, draw=red, shape=circle]
\tikzstyle{purplerdot}=[fill={rgb,255: red,128; green,0; blue,128}, draw=black, shape=circle]
\tikzstyle{greendot}=[fill={rgb,255: red,0; green,110; blue,0}, draw={rgb,255: red,0; green,110; blue,0}, shape=circle]

\tikzstyle{new edge style 0}=[-]
\tikzstyle{new edge style 1}=[-, draw={rgb,255: red,128; green,0; blue,128}]
\tikzstyle{blueline}=[-, draw=blue]
\tikzstyle{redline}=[-, draw=red]

\newcommand{\Z}{\mathbb{Z}}
\newcommand{\R}{\mathbb{R}}
\newcommand{\Q}{\mathbb{Q}}

\newcommand{\CP}{\mathbb{C}P}

\DeclareMathOperator{\Id}{Id}

\DeclareMathOperator{\GCD}{GCD}
\DeclareMathOperator{\Torsion}{Torsion}
\DeclareMathOperator{\rk}{rk}

\newtheorem{thm}{Theorem}
\newtheorem{dfn}[thm]{Definition}
\newtheorem{notation}[thm]{Notation}
\newtheorem{prop}[thm]{Proposition}
\newtheorem{lemma}[thm]{Lemma}

\newtheorem{question}[thm]{Question}
\newtheorem{addendum}{Addendum}[thm]

\newtheorem*{thm*}{Theorem}
\newtheorem*{prop*}{Proposition}

\newtheorem{thmalph}{Theorem}

\theoremstyle{remark}

\newtheorem{example}[thm]{Example}
\newtheorem*{example*}{Example}
\newtheorem{rmk}[addendum]{Remark}
\newtheorem*{remark}{Remark}

\graphicspath{ {images/} }
\title{
{Constructing Rational Homology 3-Spheres That Bound Rational Homology 4-Balls}\\
}
\author{Lisa Lokteva}
\date{\today}

\begin{document}

\maketitle

\begin{abstract}
We present three large families of new examples of plumbed $3$-manifolds that bound rational homology $4$-balls. These are constructed using two operations, also defined here, that preserve the lack of a lattice embedding obstruction to bounding rational homology balls. Apart from in the cases shown in this paper, it remains open whether these operations are rational homology cobordisms in general.

The families of new examples include a multitude of families of rational surgeries on torus knots, and we explicitly describe which positive torus knots we now know to have a surgery that bounds a rational homology ball.

While not the focus of this paper, we implicitly confirm the slice-ribbon conjecture for new, more complicated, examples of arborescent knots, including many Montesinos knots.
\end{abstract}

\section{Introduction}

In Kirby's problem 4.5 \cite{kirbylist}, Casson asks which rational homology $3$-spheres bound rational homology $4$-balls. While rational homology $3$-spheres abound in nature, including the $r$-surgery $S^3_r(K)$ on a knot $K$ for any $r \in \Q - \{ 0 \}$, very few of them actually bound rational homology balls. In fact, Aceto and Golla showed in \cite[Theorem 1.1]{acetogolla}, that for every knot $K$ and every $q \in \Z_+$, there exist at most finitely many $p \in \Z_+$ such that $S^3_{p/q}(K)$ bounds a rational homology ball. It is hard to answer Casson's question in full generality, but recently a great deal of progress has been made on specific classes of rational homology $3$-spheres. For example, in 2007 we learnt the answer for lens spaces \cite{liscasingle07, liscamultiple07}, in 2020 for positive integral surgeries on positive torus knots \cite{acetogolla, GALL}, and in between we learnt the answer for several other classes on Seifert fibred spaces with three exceptional fibres \cite{lecuonacomplementary, lecuonamontesinos}. We do not yet know the answer for general Seifert fibred spaces with three exceptional fibres. In \cite{lokteva2021surgeries}, the author started studying surgeries on algebraic (iterated torus) knots, which are not Seifert fibred but decompose into Seifert fibred spaces when cut along a maximal system of incompressible tori \cite{gordon}.

An important tool to study which $3$-manifolds bound rational homology balls is the following well-known corollary of Donaldson's diagonalisation theorem \cite[Theorem 1]{Donaldson1987}: \begin{prop}[Corollary of Donaldson's Theorem] \label{prop:latticeembeddings}
Let $Y$ be a rational homology 3-sphere and $Y=\partial X$ for $X$ a negative definite smooth connected oriented 4-manifold. If $Y=\partial W$ for a smooth rational homology 4-ball $W$, then there exists a lattice embedding \[ (H_2(X)/\Torsion, Q_{X}) \hookrightarrow (\Z^{\rk H_2(X)}, -\Id). \]
\end{prop} \noindent Here, $Q_X$ is the intersection form on $H_2(X)/\Torsion$. Determining which $3$-manifolds in a family $\mathfrak{F}$ bound rational homology $4$-balls using lattice embeddings often goes like this: \begin{enumerate}[label=(\roman*)]
\item Find a negative definite filling $X(Y)$ for every $Y \in \mathfrak{F}$.
\item Guess the family $\mathfrak{F}' \subset \mathfrak{F}$ of manifolds whose filling's intersection lattice (that is second homology with the intersection form) embeds into the standard lattice of the same rank.
\item Show that $(H_2(X(Y)),Q_{X(Y)})$ does not embed into $(\Z^{b_2(X(Y))}, -\Id)$ for any $Y \in \mathfrak{F}-\mathfrak{F}'$.
\item Hopefully prove that $Y$ bounds a rational homology ball for any $Y\in \mathfrak{F}'$.
\end{enumerate} Every step of this process has the potential to go wrong. For starters, there exist $3$-manifolds without any definite fillings \cite{gollalarson}. However, lens spaces, surgeries on torus knots and large surgeries on algebraic knots do have definite fillings. In fact, they all bound definite plumbings of disc bundles on spheres. Step (iv) is not guaranteed to work either. For example, $S^3_{-m^2}(K)$ bounds the knot trace $D^4_{-m^2}(K)$ ($D^4$ with a $-m^2$-framed $2$-handle glued along $K$) which has intersection lattice $(\Z, (-m^2))$ which embeds into $(\Z, -\Id)$, but according to \cite[Theorem 1.2]{acetogolla}, $S^3_{-m^2}(K)$ bounds a rational homology ball for at most two positive integer values of $m$. However, in \cite{liscasingle07, liscamultiple07,acetogolla,GALL} the authors managed to find a different filling $X(Y)$ for each $Y$ in such a way that the lattice embedding obstruction turned out perfect. These $X(Y)$s have been plumbings of disc bundles on spheres with a tree-shaped plumbing graph, moreover satisfying the property that the quantity \[I=\sum_{v \in V} (-w(v)-3), \] where $V$ is the set of vertices of the graph and $w(v)$ is the weight of $v$, would be negative.

Steps (ii) and (iii) can sometimes be done at the same time, but often, like in \cite{GALL} where $\mathfrak{F}$ is the set of positive integral surgeries on positive torus knots, they cannot. It is then important to eliminate embeddable cases early in order to proceed with step (iii).  Theorem 1.1 in \cite{GALL}, the classification of positive integral surgeries on positive torus knots bounding rational homology balls, lists 5 families that are Seifert fibred spaces with 3 exceptional fibres. They bound a negative definite star-shaped plumbing with three legs. Families (1)-(3) have two complementary legs, that is two legs whose weight sequences are Riemenschneider dual (defined, for the reader's convenience, in Section \ref{sec:complegs} of this paper). All such 3-manifolds that bound a rational homology ball have been classified by Lecuona in \cite{lecuonacomplementary}. Family (5) contains two exceptional graphs which were known to bound rational homology balls both because they arise as boundaries of tubular neighbourhoods of rational cuspidal curves in \cite{bobluhene07} and because they are surgeries on torus knots $T(p,q)$ where $q \equiv \pm 1 \pmod{p}$, which were studied in \cite{acetogolla}. However, Family (4) took the authors of \cite{GALL} a while to find, in the meantime thwarting their attempts at step (iii). Eventually they found Family (4) using a computer. This allowed them to finish off their lattice embedding analysis, but Family (4) still looked surprising and strange and begged the question of ``How could we have predicted its existence?"

\subsection{GOCL and IGOCL Moves}

This work came out of widening the perspective and asking which boundaries of $4$-manifolds described by plumbing trees with negative definite intersection forms and low $I$ bound rational homology balls. As a preliminary question, we asked ourselves which plumbing trees generate an embeddable intersection lattice. We looked at what the graphs of $3$-manifolds we know to bound rational homology balls look like and tried to see if there are any common patterns. In \cite[Remark 3.2]{lecuonamontesinos}, Lecuona describes how to get all lens spaces that bound rational homology balls from the linear graphs $(-2,-2,-2)$, $(-3,-2,-3,-3,-3)$, $(-3,-2,-2,-3)$ and $(-2,-2,-3,-4)$ using some modifications. (She restates Lisca's result in \cite{liscasingle07} in the language of plumbing graphs rather than fractions $p/q$ for $L(p,q)$.) In this paper we define a couple of moves called GOCL and IGOCL moves on embedded plumbing graphs that preserve embeddability and generalise the moves described by Lecuona. From this point of view, Lecuona's list simply turns into a list of IGOCL and GOCL moves that keep the graph linear. The IGOCL move was also used by Jonathan Simone in \cite{simone2020classification} under the name of expansions. The GOCL move is a generalisation of Lisca's expansions in \cite{liscasingle07}.

Now, we may ask ourselves if these moves preserve the property of the described $3$-manifolds bounding rational homology balls. There is unfortunately no obvious rational homology cobordism between two $3$-manifolds differing by a GOCL or an IGOCL move. We can however prove that repeated applications of these moves to the embeddable linear graphs $(-3,-2,-3,-3,-3)$, $(-3,-2,-2,-3)$ and $(-2,-2,-3,-4)$ give $3$-manifolds bounding rational homology balls. This results in the following theorem:

\begin{figure}
\centering
\begin{tikzpicture}
	\begin{pgfonlayer}{nodelayer}
		\node [style={small_black_dot}, label={below:$-(1+b_1)$}] (0) at (0, 6) {};
		\node [style={small_black_dot}, label={above:$-a_1$}] (1) at (-3, 6) {};
		\node [style={small_black_dot}, label={above:$-(3+k)$}] (2) at (-6, 6) {};
		\node [style={small_black_dot}, label={above:$-(1+\alpha_1)$}] (3) at (3, 6) {};
		\node [style={small_black_dot}, label={above:$-2$}] (4) at (6, 6) {};
		\node [style={small_black_dot}, label={above:$-2$}] (5) at (10, 6) {};
		\node [style={small_black_dot}, label={below:$-(1+\beta_1)$}] (6) at (13, 6) {};
		\node [style=lookingempty] (7) at (8, 6) {$\cdots$};
		\node [style=reddot, label={left:$-a_2$}] (8) at (-3, 3) {};
		\node [style=reddot, label={left:$-a_{m_1}$}] (9) at (-3, -1) {};
		\node [style=reddot, label={left:$-\alpha_2$}] (10) at (3, 3) {};
		\node [style=reddot, label={left:$-\alpha_{m_2}$}] (11) at (3, -1) {};
		\node [style=bluedot, label={left:$-b_2$}] (12) at (0, 9) {};
		\node [style=bluedot, label={left:$-b_{n_1}$}] (13) at (0, 13) {};
		\node [style=bluedot, label={left:$-\beta_2$}] (14) at (13, 9) {};
		\node [style=bluedot, label={left:$-\beta_{n_2}$}] (15) at (13, 13) {};
		\node [style=lookingempty] (16) at (-3, 1) {$\vdots$};
		\node [style=lookingempty] (17) at (3, 1) {$\vdots$};
		\node [style=lookingempty] (18) at (0, 11) {$\vdots$};
		\node [style=lookingempty] (19) at (13, 11) {$\vdots$};
	\end{pgfonlayer}
	\begin{pgfonlayer}{edgelayer}
		\draw (2) to (1);
		\draw (1) to (0);
		\draw (0) to (3);
		\draw (3) to (4);
		\draw (4) to (7);
		\draw (7) to (5);
		\draw (5) to (6);
		\draw [style=redline] (1) to (8);
		\draw [style=redline] (8) to (16);
		\draw [style=redline] (16) to (9);
		\draw [style=redline] (3) to (10);
		\draw [style=redline] (10) to (17);
		\draw [style=redline] (17) to (11);
		\draw [style=blueline] (0) to (12);
		\draw [style=blueline] (12) to (18);
		\draw [style=blueline] (18) to (13);
		\draw [style=blueline] (6) to (14);
		\draw [style=blueline] (14) to (19);
		\draw [style=blueline] (19) to (15);
	\end{pgfonlayer}
\end{tikzpicture}
\caption{These are the graphs obtainable by performing IGOCL and GOCL moves on the linear graph $(-3,-2,-3,-3,-3)$. Here the length of the chain of $-2$'s is $k\geq 0$, $(a_1,\dots,a_{m_1})$ and $(\alpha_1,\dots,\alpha_{m_2})$ are complementary sequences, and $(b_1,\dots,b_{n_1})$ and $(\beta_1,\dots,\beta_{n_2})$ are complementary sequences.} \label{fig:32333general}
\end{figure}

\begin{figure}
\centering
\begin{tikzpicture}
	\begin{pgfonlayer}{nodelayer}
		\node [style={small_black_dot}, label={left:$1-\beta_1-\zeta_1$}] (0) at (-14, 2) {};
		\node [style={small_black_dot}, label={above:$-a_1$}] (1) at (-8, 2) {};
		\node [style={small_black_dot}, label={above:$-b_1$}] (2) at (-2, 2) {};
		\node [style={small_black_dot}, label={right:$1-\alpha_1-z_1$}] (3) at (4, 2) {};
		\node [style=reddot, label={left:$-a_2$}] (4) at (-8, 0) {};
		\node [style=reddot, label={left:$-a_{l_1}$}] (5) at (-8, -4) {};
		\node [style=reddot, label={left:$-\alpha_2$}] (6) at (4, 0) {};
		\node [style=reddot, label={left:$-\alpha_{l_2}$}] (7) at (4, -4) {};
		\node [style=greendot, label={left:$-\beta_2$}] (8) at (-14, 0) {};
		\node [style=greendot, label={left:$-\beta_{m_2}$}] (9) at (-14, -4) {};
		\node [style=greendot, label={left:$-b_2$}] (10) at (-2, 0) {};
		\node [style=greendot, label={left:$-b_{m_1}$}] (11) at (-2, -4) {};
		\node [style=bluedot, label={left:$-z_2$}] (12) at (4, 4) {};
		\node [style=bluedot, label={left:$-z_{n_1}$}] (13) at (4, 8) {};
		\node [style=bluedot, label={left:$-\zeta_2$}] (14) at (-14, 4) {};
		\node [style=bluedot, label={left:$-\zeta_{n_2}$}] (15) at (-14, 8) {};
		\node [style=lookingempty] (16) at (-14, 6) {$\vdots$};
		\node [style=lookingempty] (17) at (4, 6) {$\vdots$};
		\node [style=lookingempty] (18) at (-14, -2) {$\vdots$};
		\node [style=lookingempty] (19) at (-8, -2) {$\vdots$};
		\node [style=lookingempty] (20) at (-2, -2) {$\vdots$};
		\node [style=lookingempty] (21) at (4, -2) {$\vdots$};
		\node [style=none] (22) at (-14, 0) {};
	\end{pgfonlayer}
	\begin{pgfonlayer}{edgelayer}
		\draw [style=new edge style 0] (0) to (1);
		\draw [style=new edge style 0] (1) to (2);
		\draw [style=new edge style 0] (2) to (3);
		\draw [style=new edge style 0] (0) to (8);
		\draw [style=new edge style 0] (8) to (18);
		\draw [style=new edge style 0] (18) to (9);
		\draw [style=new edge style 0] (1) to (4);
		\draw [style=new edge style 0] (4) to (19);
		\draw [style=new edge style 0] (19) to (5);
		\draw [style=new edge style 0] (2) to (10);
		\draw [style=new edge style 0] (10) to (20);
		\draw [style=new edge style 0] (20) to (11);
		\draw [style=new edge style 0] (3) to (6);
		\draw [style=new edge style 0] (6) to (21);
		\draw [style=new edge style 0] (21) to (7);
		\draw [style=new edge style 0] (0) to (14);
		\draw [style=new edge style 0] (14) to (16);
		\draw [style=new edge style 0] (16) to (15);
		\draw [style=new edge style 0] (3) to (12);
		\draw [style=new edge style 0] (12) to (17);
		\draw [style=new edge style 0] (17) to (13);
	\end{pgfonlayer}
\end{tikzpicture}
\caption{The form of all graphs obtainable from $(-3,-2,-2,-3)$ using GOCL moves. Here the sequences $(a_1,\dots,a_{l_1})$ and $(\alpha_1,\dots,\alpha_{l_2})$ are complementary, as well as the sequences $(b_1,\dots, b_{m_1})$ and $(\beta_1,\dots,\beta_{m_2})$, and the sequences $(z_1,\dots, z_{n_1})$ and $(\zeta_1,\dots,\zeta_{n_2})$.} \label{fig:3223general}
\end{figure}

\begin{figure}
\centering
\begin{tikzpicture}
	\begin{pgfonlayer}{nodelayer}
		\node [style={small_black_dot}, label={left:$-(2+k)$}] (0) at (-3, 5) {};
		\node [style={small_black_dot}, label={right:$-b_1$}] (1) at (-3, 3) {};
		\node [style={small_black_dot}, label={left:$-1-a_1$}] (2) at (-3, 1) {};
		\node [style={small_black_dot}, label={below:$-\alpha_1-\beta_1$}] (3) at (-3, -7.75) {};
		\node [style=reddot, label={above:$-a_2$}] (4) at (-1, 1) {};
		\node [style=reddot, label={above:$-a_3$}] (5) at (1, 1) {};
		\node [style=lookingempty] (6) at (2.5, 1) {$\cdots$};
		\node [style=reddot, label={above:$-a_{n_1}$}] (7) at (4, 1) {};
		\node [style=reddot, label={above:$-\alpha_2$}] (8) at (-1, -7.75) {};
		\node [style=lookingempty] (9) at (0.5, -7.75) {$\cdots$};
		\node [style=reddot, label={above:$-\alpha_{n_2}$}] (10) at (2, -7.75) {};
		\node [style=bluedot, label={above:$-b_2$}] (11) at (-5, 3) {};
		\node [style=lookingempty] (12) at (-7, 3) {$\cdots$};
		\node [style=bluedot, label={above:$-b_{m_1}$}] (13) at (-9, 3) {};
		\node [style=bluedot, label={above:$-\beta_2$}] (14) at (-5, -7.75) {};
		\node [style=lookingempty] (15) at (-7, -7.75) {$\cdots$};
		\node [style=bluedot, label={above:$-\beta_{m_2}$}] (16) at (-9, -7.75) {};
		\node [style={small_black_dot}, label={right:$-2$}] (18) at (-3, -1) {};
		\node [style=lookingempty] (19) at (-3, -3) {$\vdots$};
		\node [style={small_black_dot}, label={right:$-2$}] (20) at (-3, -5) {};
	\end{pgfonlayer}
	\begin{pgfonlayer}{edgelayer}
		\draw [style=new edge style 0] (0) to (1);
		\draw [style=new edge style 0] (1) to (2);
		\draw (2) to (4);
		\draw (4) to (5);
		\draw (5) to (6);
		\draw (6) to (7);
		\draw (3) to (8);
		\draw (8) to (9);
		\draw (9) to (10);
		\draw (13) to (12);
		\draw (12) to (11);
		\draw (11) to (1);
		\draw (14) to (3);
		\draw (15) to (14);
		\draw (16) to (15);
		\draw (2) to (18);
		\draw (18) to (19);
		\draw (19) to (20);
		\draw (20) to (3);
	\end{pgfonlayer}
\end{tikzpicture}
\caption{The form of all graphs obtainable from $(-2,-2,-3,-4)$ by GOCL and IGOCL moves. Here the length of the chain of $-2$'s is $k\geq 0$, $(a_1,\dots,a_{m_1})$ and $(\alpha_1,\dots,\alpha_{m_2})$ are complementary sequences, and $(b_1,\dots,b_{n_1})$ and $(\beta_1,\dots,\beta_{n_2})$ are complementary sequences.} \label{fig:liscagraph2234extendedgeneral}
\end{figure}

\begin{thmalph} \label{thm:general}
All $3$-manifolds described by the plumbing graphs in Figures \ref{fig:32333general}, \ref{fig:3223general} and \ref{fig:liscagraph2234extendedgeneral} bound rational homology balls.
\end{thmalph}

\begin{remark}
See Definition \ref{def:complementary} for the definition of \textit{complementary}.
\end{remark}

\noindent There are several methods to prove this theorem. The easiest one, which we can call \textbf{The Simple Method}, uses the following proposition, which follows from the long exact sequence of the pair combined with Poincaré duality and the Universal Coefficient Theorem:

\begin{prop}[The Simple Method] \label{prop:Xhas10n1n2handles}
If a $4$-manifold $X$ consists of one $0$-handle, $n$ $1$-handles, and $n$ $2$-handles, and if $\partial X$ is a rational homology $S^3$, then $X$ is a rational homology ball.
\end{prop}

\noindent We use this by showing that we may perform two integral surgeries on the three-manifolds described by these plumbing trees and obtain $(S^1 \times S^2)\# (S^1 \times S^2)$. This method works for every one of the families of Figures \ref{fig:32333general}, \ref{fig:3223general} and \ref{fig:liscagraph2234extendedgeneral}.

A more refined method is to show that the above plumbed $3$-manifolds are double covers of $S^3$ branched over a $\chi$-slice link, that is a link bounding a surface $S$ of Euler characteristic $1$ in $D^4$ \cite[Definition 1]{donaldowens}. By \cite[Proposition 5.1]{donaldowens}, the double cover of $D^4$ branched over a surface of Euler characteristic $1$ is a rational homology $4$-ball. We use this method, described in Subsection \ref{subsec:complicatedmethod}, and more specifically Proposition \ref{prop:complicatedmethod}, for the graphs in Figures \ref{fig:32333general} and \ref{fig:3223general}. Given the way we construct $S$ in this paper, this method amounts to The Simple Method, but with the extra step of showing that we can perform the surgeries equivariantly under an involution. This extra step is quite challenging, and we do not at this moment know if we can perform it on the family of Figure \ref{fig:liscagraph2234extendedgeneral}. The bonus of using this, more difficult, method is that we obtain new examples of links that are $\chi$-ribbon in the process.

The families of Figures \ref{fig:32333general}, \ref{fig:3223general} and \ref{fig:liscagraph2234extendedgeneral}, together with the one generated by GOCL moves from $(-2,-2,-2)$, include all lens spaces bounding rational homology balls. The $(-2,-2,-2)$ family, however, only includes linear graphs already found by Lisca. On the other hand, the families of Figures \ref{fig:32333general}, \ref{fig:3223general} and \ref{fig:liscagraph2234extendedgeneral} also contain more complicated graphs. In \cite{aceto20}, Aceto defines linear complexity of a plumbing tree to be the minimal number of vertices we need to remove in order to get a linear graph. Our families have linear complexities up to 2. Many papers, e.g. \cite{aceto20, acetogolla, GALL, lecuonamontesinos, simone2020classification}, that use lattice embeddings to obstruct plumbed $3$-manifolds from bounding a rational homology ball have used arguments of the form ``If my graph $\Gamma$ is embeddable, then this other linear graph obtained from $\Gamma$ is embeddable, and we know what those look like.", which gets harder to do the further $\Gamma$ is away from being linear. Thus, we do not yet really have lattice embedding obstructions for families of graphs of complexity greater than $1$. The families of Theorem \ref{thm:general} include many graphs of Seifert fibred spaces. They include Family (4) in \cite{GALL} and predict its existence because Family (4) is just the intersection between the set of graphs in Figures \ref{fig:32333general}, \ref{fig:3223general} and \ref{fig:liscagraph2234extendedgeneral} and the negative definite plumbing graphs of positive integral surgeries on positive torus knots.

As mentioned above, there is no obvious rational homology cobordism between the $3$-manifolds described by two plumbing graphs differing by a GOCL or an IGOCL move. This is interesting in comparison with the case in the works by Aceto \cite{aceto20} and Lecuona \cite{lecuonacomplementary}. Lecuona shows that given a plumbing graph $\Gamma$, you can modify it to a graph $\Gamma'$ by subtracting $1$ from the weight of a vertex $v$ and attaching two complementary legs $(-a_1,\dots,-a_m)$ and $(-b_1,\dots,-b_n)$ (see Section \ref{sec:complegs} or \cite{lecuonacomplementary} for definitions) to $v$, and the $3$-manifolds $Y_\Gamma$ and $Y_{\Gamma'}$ described by the graphs will be rational homology cobordant, that is bound a rational homology $4$-ball if and only if the other one does. Thus, if she wants to know if a $Y_{\Gamma'}$, for $\Gamma'$ a graph with two complementary legs coming out of the same vertex, bounds a rational homology ball, she can reduce it to the same question for a simpler graph. However, since we do not know if the GOCL and IGOCL moves are rational homology cobordisms, we cannot play this trick for complementary legs growing out of different vertices.

Another work that has shown that applying GOCL moves to embedded plumbing graphs bounding certain rational homology balls gives us new plumbed $3$-manifolds that bound rational homology balls is \cite{akbulutlarson18} by Akbulut and Larson. They show that the families $\Sigma(2, 4n+1, 12n+5)$ and $\Sigma(3,3n+1, 12n+5)$ of Brieskorn spheres bound rational homology balls. In fact, these families are obtained by applying GOCL moves to the plumbing graphs of $\Sigma(2,5,17)$ and $\Sigma(3,4,17)$. Just like us in Section \ref{sec:analysis}, they perform a surgery on their spaces and the result of this surgery is the same for each space, in their case a $0$-surgery on the figure-eight knot. Their lemma \cite[Lemma 2]{akbulutlarson18}, saying that any integral homology sphere obtained from a surgery on a $0$-surgery on a rationally slice knot, can be viewed as a generalisation of The Simple Method and Proposition \ref{prop:Xhas10n1n2handles} when $n=1$. With the same technique as Akbulut and Larson, \c{S}avk \cite{Savk2020} constructed two more families $\Sigma(2, 4n+3, 12n+7)$ and $\Sigma(3,3n+2, 12n+7)$ of Brieskorn spheres that bound rational homology balls. These are obtainable from $\Sigma(2,7,19)$ and $\Sigma(3,5,19)$, respectively, using IGOCL moves.

\subsection{Rational Surgeries on Torus Knots}

An interesting generalisation of \cite[Theorem 1.1]{GALL}, would be to classify all positive rational surgeries on positive torus knots that bound rational homology balls. Theorem \ref{thm:general} allows us to construct more examples of such surgeries than is sightly to write down. Instead, we may ask ourselves the following question: \begin{question}
For which $1<p<q$ with $\GCD(p,q)=1$ is there an $r\in \Q_+$ such that $S^3_r(T(p,q))$ bounds a rational homology ball?
\end{question} Section \ref{sec:torusknots} is dedicated to showing the following theorem: \begin{thmalph} \label{thm:torusknotlist}
For the following pairs $(p,q)$ with $1<p<q$ and $\GCD(p,q)=1$, there is at least one $r \in \Q_+$ such that $S^3_r(T(p,q))$ bounds a rational ball. Here $k,l \geq 0$. \begin{enumerate}
\item $(k+2,(l+1)(k+2)+1)$ \label{enum:onemodp}
\item $(k+2,(l+2)(k+2)-1)$ \label{enum:minusonemodp}
\item $(2k+3,(l+1)(2k+3)+2)$ \label{enum:twomodp}
\item $(2k+3,(l+2)(2k+3)-2)$ \label{enum:minustwomodp}
\item $(k^2+7k+11, k^3+12k^2+45k+51)$ \label{enum:32333cubic}

\item $(S^{(k)}_{l+1},S^{(k)}_{l+2})$ for $(S^{(k)}_i)$ a sequence defined by $S^{(k)}_0=1$, $S^{(k)}_1=2$, $S^{(k)}_2=2k+7$ and $S^{(k)}_{i+2}=(k+4)S^{(k)}_{i+1}-S^{(k)}_i$. \label{enum:3223one}
\item $(T^{(k)}_{l+1},T^{(k)}_{l+2})$ for $(T^{(k)}_i)$ a sequence defined by $T^{(k)}_0=1$, $T^{(k)}_1=k+2$, $T^{(k)}_2=k^2+6k+7$ and $T^{(k)}_{i+2}=(k+4)T^{(k)}_{i+1}-T^{(k)}_i$. \label{enum:3223two}
\item $(U_{l+1},U_{l+2})$ for $(U_i)$ a sequence defined by $U_0=1$, $U_1=3$, $U_2=14$ and $U_{i+2}=5U_{i+1}-U_i$. \label{enum:32333recursive}

\item $(R_{l+1},R_{l+2})$ for $(R_i)$ a sequence defined by $R_0=1$, $R_1=3$, $R_2=17$ and $R_{i+2}=6R_{i+1}-R_i$. \label{enum:2234gall}

\item $(P_{l+1},P_{l+2})$ for $(P_i)$ a sequence defined by $P_0=1$, $P_1=4$, $P_2=19$ and $P_{i+2}=5P_{i+1}-P_i$. \label{enum:2234one}
\item $(Q_{l+1},Q_{l+2})$ for $(Q_i)$ a sequence defined by $Q_0=1$, $Q_1=2$, $Q_2=9$ and $Q_{i+2}=5Q_{i+1}-Q_i$. \label{enum:2234two}

\item $(A, (n+1)Q+P)$ for $P$ and $Q$ such that $L(Q,P)$ bounds a rational homology ball (or equivalently $\frac{Q}{P}$ lying in Lisca's set $\mathcal{R}$ \cite{liscasingle07}), and $A$ a multiplicative inverse to either $Q$ or $nQ+P$ modulo $(n+1)Q+P$ such that $0<A<(n+1)Q+P$. \label{enum:compllegs1}
\item $((n+1)Q+P, (l+1)((n+1)Q+P)+A)$ for $P$ and $Q$ such that $L(Q,P)$ bounds a rational homology ball (or equivalently $\frac{Q}{P}$ lying in Lisca's set $\mathcal{R}$ \cite{liscasingle07}), and $A$ a multiplicative inverse to either $Q$ or $nQ+P$ modulo $(n+1)Q+P$ such that $0<A<(n+1)Q+P$. \label{enum:compllegs2}
\item $(B, P)$ for $P$ and $Q$ such that $L(Q,P)$ bounds a rational homology ball (or equivalently $\frac{Q}{P}$ lying in Lisca's set $\mathcal{R}$ \cite{liscasingle07}), and $B$ a multiplicative inverse to either $P\lceil \frac{Q}{P}\rceil-Q$ or $Q-P\lfloor \frac{Q}{P} \rfloor$ modulo $P$ such that $0<B<P$. \label{enum:compllegs3}

\item $(P, (l+1)P+B)$ for $P$ and $Q$ such that $L(Q,P)$ bounds a rational homology ball (or equivalently $\frac{Q}{P}$ lying in Lisca's set $\mathcal{R}$ \cite{liscasingle07}), and $B$ a multiplicative inverse to either $P\lceil \frac{Q}{P}\rceil-Q$ or $Q-P\lfloor \frac{Q}{P} \rfloor$ modulo $P$ such that $0<B<P$. \label{enum:compllegs4}
\item $(P,Q)$ such that there is a number $n$ such that $(P,Q,n) \in \mathcal{R} \sqcup \mathcal{L}$ for the sets $\mathcal{R}$ and $\mathcal{L}$ defined in \cite[Theorem 1.1]{GALL}. (Note that here $r=n \in \{PQ, PQ-1, PQ+1 \}$, so we are looking at an integral surgery.) \label{enum:nearpq}
\end{enumerate}
\end{thmalph}

\noindent The curious reader can use the methods of Section \ref{sec:torusknots} to obtain the surgery coefficients $r$ too.

In Theorem \ref{thm:torusknotlist}, case \ref{enum:nearpq} is shown to bound rational homology balls in \cite{GALL} and reflects the degenerate cases of surgeries on torus knots that are lens spaces or connected sums of lens spaces, cases \ref{enum:compllegs1}-\ref{enum:compllegs4} are shown to bound rational homology balls in \cite{lecuonacomplementary} because their graphs have a pair of complementary legs, while the cases \ref{enum:onemodp}-\ref{enum:2234two} are shown to bound rational homology balls in this paper, using that there exists an $r$ such that $S^3_r(T(p,q))$ bounds a graph in the families of Figures \ref{fig:32333general}, \ref{fig:3223general} and \ref{fig:liscagraph2234extendedgeneral}. The authors of \cite{GALL} classified all positive integral surgeries on positive torus knots that bound rational homology balls. The classification included 18 families, whereof families (6)-(18) are included in our family \ref{enum:nearpq}, family (4) in our family \ref{enum:2234gall}, and the others in families \ref{enum:onemodp} and \ref{enum:minusonemodp}.

At the moment of writing we do not know of any other positive torus knots having positive surgeries bounding rational homology balls. The pair $(8,19)$ is in some metric the smallest example not to appear on the list of Theorem \ref{thm:torusknotlist}. Thus we may concretely ask: \begin{question}
Is there an $r \in \Q_+$ such that $S^3_r(T(8,19))$ bounds a rational homology ball?
\end{question}

We may also note that some positive torus knots have many surgeries that bound rational homology balls. For example, Theorem \ref{thm:general} allows us to construct numerous finite and infinite families of surgery coefficients $r \in \Q_+$ such that $S^3_r(T(2,3))$ bounds a rational homology ball. All we need to do is to choose weights in the graphs in Figures \ref{fig:liscagraph2234extendedgeneral}, \ref{fig:3223general} and \ref{fig:32333general} so that we get a starshaped graph with three legs whereof one is $(-2)$ and another is either $(-2,-2)$ or $(-3)$. For example $S^3_{\frac{(11k+20)^2}{22k^2+79k+71}}(T(2,3))$ bounds a plumbing of the shape in Figure \ref{fig:liscagraph2234extendedgeneral} with $(b_1,\dots, b_{m_1})=(4)$ and $(a_1,\dots,a_{n_1})=(3)$, and thus bounds a rational homology ball for any $k \geq 0$. There are also surgeries on $T(2,3)$ that bound rational balls, but do not have graphs of the shapes of Figures \ref{fig:32333general}, \ref{fig:3223general} or \ref{fig:liscagraph2234extendedgeneral}. For example, we have $S^3_{\frac{64}{7}}(T(2,3))=-S^3_{64}(T(3,22))$, which bounds a rational homology ball because it is the boundary of the tubular neighbourhood of a rational curve in $\CP^2$ \cite{bobluhene07}, but whose lattice embedding contains a basis vector with coefficient $2$, which we do not get by applying $GOCL$ or $IGOCL$ moves to $(-2,-2,-3,-4)$, $(-3,-2,-2,-3)$ and $(-3,-2,-3,-3,-3)$. The lattice embedded plumbing graph of $S^3_{\frac{64}{7}}(T(2,3))$ does however fit into Family $\mathcal{C}$ of \cite{stipszabwahl} of symplectically embeddable plumbings. Unfortunately, Family $\mathcal{C}$ of \cite{stipszabwahl} contains both surgeries on $T(2,3)$ that bound rational homology balls and ones that do not. For example, $S^3_{\frac{169}{25}} (T(2,3))$ of \cite[Section 2.4, Figure 12]{stipszabwahl} does not bound a rational ball despite bounding a plumbing with an embeddable intersection form. A later paper \cite{bhupalstipsicz} classified which surgeries on $T(2,3)$ appearing in Family $\mathcal{C}$, viewed as surface singularity links, bound a rationally acyclic Milnor fibre. Interestingly, all but two of the embedded graphs in that family are generated by applying IGOCL moves to the graph of $S^3_{\frac{64}{7}}(T(2,3))$. However, we do not know if any other members of Family $\mathcal{C}$ bound a rational homology ball which is not a Milnor fibre. Hence, the following is a rich open question worth studying:

\begin{question}
For which $r \in \Q_+$ does $S^3_r(T(2,3))$ bound a rational homology ball?
\end{question}

\subsection{Outline}
We start off the paper with Section \ref{sec:complegs} by recalling some results on complementary legs and the basics of the lattice embedding setup. In Section \ref{sec:analysis} we define the GOCL and IGOCL moves and show Theorem \ref{thm:general}. In Section \ref{sec:torusknots} we prove Theorem \ref{thm:torusknotlist} for the families \ref{enum:onemodp}-\ref{enum:2234two}, while the other families follow directly from \cite{lecuonacomplementary} and \cite{GALL}.

\subsection*{Acknowledgements}
This research was carried out by the author when she was a PhD student at the University of Glasgow, funded by the College of Science and Engineering. This is an updated version of a part of her thesis.

\section{Complementary Legs and Lattice Embeddings} \label{sec:complegs}

In this section we list some definitions and easy propositions that are helpful to understanding the paper. We recall the definition of lattice embeddings and apply it to plumbing graphs and complementary legs. In this section, we assume that the reader is familiar with plumbings of disc bundles over spheres and how to convert plumbing diagrams into Kirby diagrams. (If not, see \cite[Example 4.6.2]{gompfstipsicz}.)

\begin{notation}\label{notation:plumbinggraph}

Given a forest-shaped plumbing graph $\Gamma$ with weight function $W: V \to \Z$, we may associate to it a $4$-manifold $X_\Gamma$ by describing its Kirby diagram. First, we draw a small unknot at each vertex of $\Gamma$. Then, for each edge, we create a Hopf linking between the knots corresponding to the edge ends as in Figure \ref{fig:plumbinggraphkirby}. We denote the resulting link by $L_\Gamma$. Then $X_\Gamma$ is the simply-connected $4$-manifold obtained by attaching $2$-handles with framing $W(v)$ to the unknot at each vertex $v$.

We denote the $3$-manifold $\partial X_\Gamma$ by $Y_\Gamma$.
\end{notation}

\begin{figure}[h]
\centering
\includegraphics[width=0.5\textwidth]{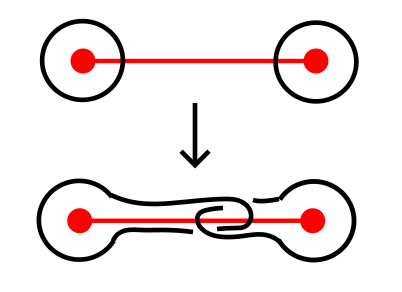}
\caption{Kirby diagram associated to a plumbing.}
\label{fig:plumbinggraphkirby}
\end{figure}

\begin{rmk}
A common abuse of terminology is ``the plumbing graph $\Gamma$ bounds a rational homology ball'', which means that $Y_\Gamma$ bounds a rational homology ball.
\end{rmk}


Let $\Gamma$ be a forest-shaped plumbing graph. The second homology of $X_\Gamma$ is the free abelian group $\Z\langle V_1,\dots,V_k \rangle $ on the vertices and the intersection form is \[ \langle V_i, V_j \rangle_{Q_{X}} =  \begin{cases}
    \text{weight of } V_i & \text{ if } i=j \\
    1 & \text{ if } V_i \text{ is adjacent to } V_j \\
    0 & \text{ otherwise.}
\end{cases} \] 

\begin{dfn} Let $X$ be a $4$-manifold with boundary. A \emph{lattice embedding} \[ f: (H_2(X)/\Torsion, Q_{X}) \hookrightarrow (\Z^{N}, -\Id) \] is a linear map $f$ such that $\langle V_i, V_j \rangle_{Q_{X}} = \langle f(V_i), f(V_j) \rangle_{-\Id}$. We will simply denote $\langle \cdot, \cdot \rangle := \langle \cdot, \cdot \rangle_{-\Id}$. If nothing else is specified, then $N=\rk H_2(X)$, that is the number of vertices in the graph. \end{dfn} \noindent Common abuses of notation include ``embedding of the graph", meaning an embedding of the lattice $(H_2(X)/\Torsion, Q_{X})$, where $X=X_\Gamma$ for a plumbing graph $\Gamma$.

Knowing when a lattice embedding exists is useful because of Proposition \ref{prop:latticeembeddings} in the introduction.


Now we turn our heads to lattice embeddings of specific plumbing graphs, namely pairs of complementary legs.

\begin{dfn} We define the negative continued fraction $[a_1,\dots, a_n]^-$ as
\[ [a_1,\dots, a_n]^-=a_1-\cfrac{1}{a_2-\cfrac{1}{\ddots - \cfrac{1}{a_n.}}} \]
\end{dfn}

Negative continued fractions often show up in low-dimensional topology because of the slam-dunk Kirby move \cite[Figure 5.30]{gompfstipsicz}, which allows us to substitute a rational surgery on a knot by an integral surgery on a link.

\begin{dfn} \label{def:complementary}
A two-component weighted linear graph $(-\alpha_1,\dots,-\alpha_n), (-\beta_1,\dots,-\beta_k)$ (with $\alpha_i, \beta_j$ integers greater than or equal to 2) is called a pair of complementary legs if \[ \frac{1}{[\alpha_1,\dots,\alpha_n]^-}+\frac{1}{[\beta_1,\dots,\beta_k]^-}=1. \] We call the sequence $(\beta_1,\dots,\beta_k)$ the Riemenschneider dual or complement of the sequence $(\alpha_1,\dots,\alpha_n)$, and we call the fractions $[\alpha_1,\dots,\alpha_n]^-$ and $[\beta_1,\dots,\beta_k]^-$ complementary.
\end{dfn}

\begin{dfn}
A Riemenschneider diagram is a finite set of points $S$ in $\Z_+ \times \Z_-$ such that $(1,-1) \in S$ and for every point $(a,b) \in S$ but one, exactly one of $(a+1,b)$ or $(a,b-1)$ is in $S$. If $(n,k) \in S$ is the point with the largest $n-k$, we say that the Riemenschneider diagram represents the fractions $[\alpha_1,\dots,\alpha_n]^-$ and $[\beta_1,\dots,\beta_k]^-$, where $\alpha_i$ is one more than the number of points with $x=i$ and $\beta_j$ is one more than the number of points with $y=-j$.
\end{dfn}

\begin{figure}[h]
\centering
\includegraphics[width=0.5\textwidth]{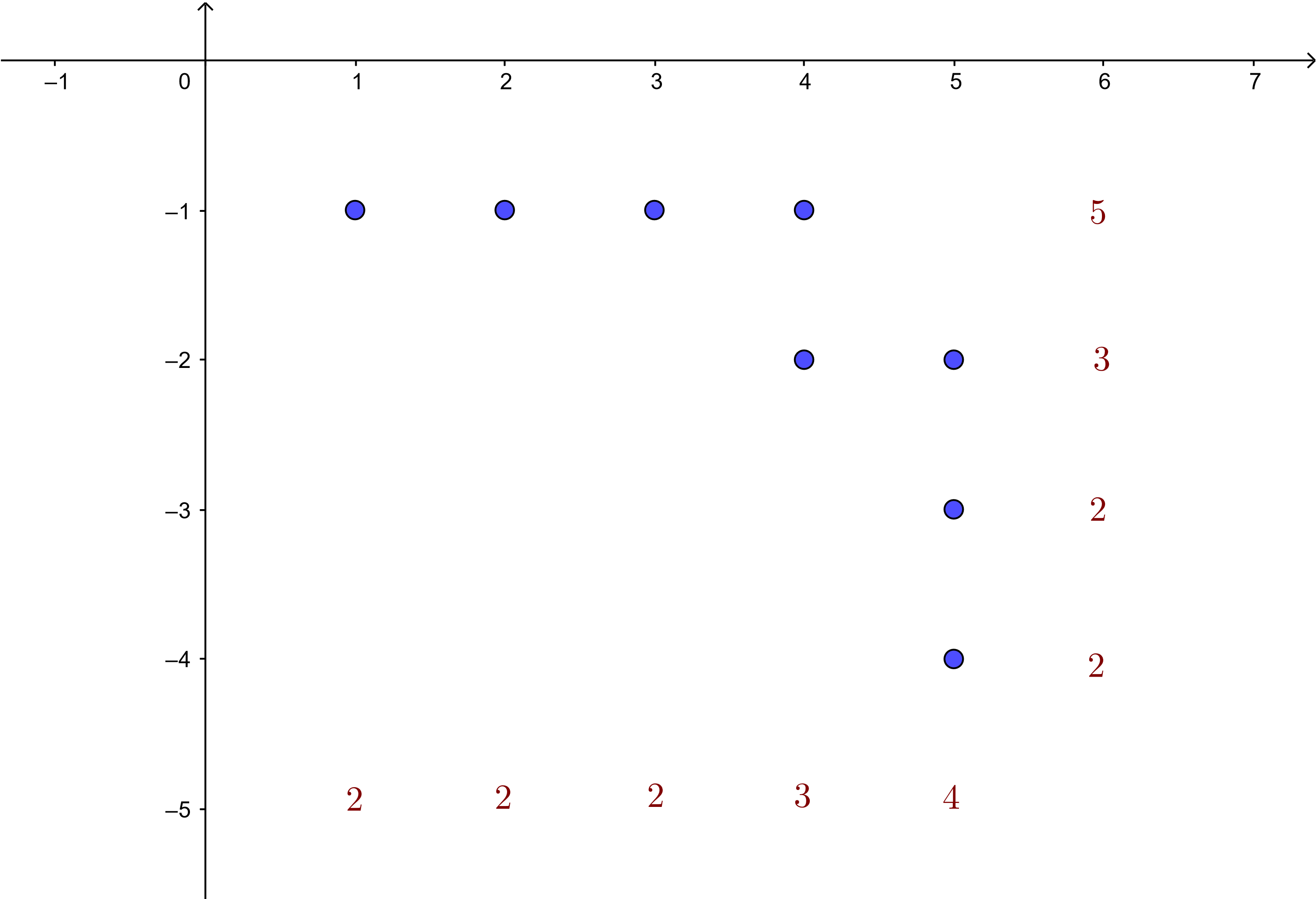}
\caption{Example of a Riemenschneider diagram representing the complementary fractions $[5,3,2,2]^-$ and $[2,2,2,3,4]^-$.}
\label{riemenschneider_diagram}
\end{figure}

\begin{example}
See Figure \ref{riemenschneider_diagram} for an example of a Riemenschneider diagram.
\end{example}

\begin{prop}[Riemenschneider \cite{riemenschneider}]
The two fractions represented by a Riemenschneider diagram are complementary.
\end{prop}

\begin{rmk}
Note that given any continued fraction $[\alpha_1,\dots,\alpha_n]^-$ with all $\alpha_i \geq 2$, we may construct a Riemenschneider diagram representing $[\alpha_1,\dots,\alpha_n]^-$ and its Riemenschneider dual.
\end{rmk}

The following theorem is well-known, but we explicitly write out the embedding construction for the reader's convienience.

\begin{prop} \label{prop:complegembed}
Every pair of complementary legs has a lattice embedding.
\end{prop}

\begin{proof}
The embedding can be constructed algorithmically from a Riemenschneider diagram. Denote the vertices of the two complementary legs by $(U_1,\dots, U_{m_1})$ and $(V_1,\dots,V_{m_2})$. These vertices generate the second homology of the plumbed $4$-manifold described by the graph. We need to send every vertex to an element of $\Z\langle e_1, \dots, e_{m_1+m_2} \rangle$. We will construct this embedding recursively through ``partial embeddings'', which are maps $f: \{U_1, \dots, U_{m_1}, V_1, \dots, V_{m_2}\} \to \Z\langle e_1, \dots, e_{m_1+m_2} \rangle $ such that $\langle f(X), f(X) \rangle \geq \text{ weight of } X$.

Order the points in the Riemenschneider diagram so that $P_1=(1,-1)$, and if $P_i=(a,b)$, then point $P_{i+1}$ is either $(a+1,b)$ or $(a,b-1)$. Now, we recursively build an embedding as follows.

\begin{itemize}
\item Start by mapping both $U_1$ and $V_1$ to $e_1$. 
\item For each non-final $i$, if the current partial embedding is $(u_1,\dots,u_n), (v_1,\dots,v_k)$ (meaning that $(U_1,\dots, U_n)$ gets mapped to $(u_1,\dots,u_n)$ and $(V_1,\dots, V_k)$ gets mapped to $(v_1, \dots, v_k)$) and $P_i=(a,b)$ is such that $P_{i+1}=(a+1,b)$, then the new partial embedding will be $(u_1,\dots,u_n+e_{i+1}), (v_1,\dots,v_k-e_{i+1},e_{i+1})$. If $P_{i+1}=(a,b-1)$, then the new partial embedding will be $(u_1,\dots,u_n-e_{i+1},e_{i+1}), (v_1,\dots,v_k+e_{i+1})$ instead.
\item If $P_i$ is final and the current partial embedding is $(u_1,\dots,u_n), (v_1,\dots,v_k)$, the new embedding will be $(u_1,\dots,u_n+e_{i+1}), (v_1,\dots,v_k-e_{i+1})$. (Or the other way around, as this sign choice is arbitary.)
\end{itemize}

It is easy to see that an embedding $(u_1,\dots,u_{m_1}), (v_1,\dots,v_{m_2})$ constructed this way will have the properties:
\begin{itemize}
    \item Each $u_i$ for $i=1,\dots,n-1$ and $v_j$ for $j=1,\dots,k$ will be a sum of consecutive basis vectors, all but the last one with coefficient $1$, and the last one with coefficient $-1$. Meanwhile, $u_n$ will be a sum of consecutive basis vectors, all with coefficient $1$.
    \item If the Riemenschneider diagram represents the fractions $[\alpha_1,\dots,\alpha_n]^-$ and $[\beta_1,\dots,\beta_k]^-$, then $\langle u_i, u_i \rangle=-\alpha_i$ and $\langle v_j, v_j \rangle=-\beta_j$.
    \item Since $u_i$ and $u_{i+1}$ have exactly one basis vector in common, one with a positive coefficient and one with a negative one, $\langle u_i, u_{i+1} \rangle =1$, and similarly $\langle v_i, v_{i+1} \rangle =1$.
    \item The other pairs $(u_i, u_j)$ (with $|i-j|>1$) don't share basis vectors and are thus orthogonal. Similarly, the pairs $(v_i, v_j)$ with $|i-j|>1$ don't share basis vectors and are thus orthogonal.
    \item It is easy to show by induction on the construction that $\langle u_i, v_j \rangle=0$ for all $i,j$.
\end{itemize} These properties show that we are in fact looking at a lattice embedding of the complementary legs.
\end{proof}

\begin{rmk}
In fact, if $e_1$ is fixed to hit the first vertex of each complementary leg, the rest of the embedding is unique up to renaming of elements and sign of the coefficient \cite[Lemma 5.2]{bakerbucklecuona}.
\end{rmk}

The following facts are useful when dealing with lattice embeddings. We will often use these properties without citing them. The first fact follows from reversing the Riemenschneider diagram, the second from embedding the sequences $(a_m,\dots,a_1)$ and $(b_n,\dots, b_1)$ as in Proposition \ref{prop:complegembed} and mapping the $-1$-weighted vertex to $-e_1$, and the rest from looking at a Riemenscheider diagram.

\begin{prop} \label{prop:latticeembeddingproperties}
Let $(a_1,\dots,a_m)$ and $(b_1,\dots, b_n)$ be complementary sequences. Then the following hold: \begin{enumerate}
\item The sequences $(a_m,\dots,a_1)$ and $(b_n,\dots, b_1)$ are complementary.
\item The linear graph $(-a_1,\dots,-a_m,-1,-b_n,\dots,-b_1)$ embeds in $(\Z^{m+n}, -\Id)$.
\item Either $a_m$ or $b_n$ must equal $2$, so assume without loss of generality that $b_n=2$. Blowing down the $-1$ in the linear graph $(-a_1,\dots,-a_m,-1,-b_n,\dots,-b_1)$ gives us the linear graph $(-a_1,\dots,-(a_m-1),-1,-b_{n-1},\dots,-b_1)$. This graph is once again a pair of complementary legs linked by a $-1$, described by the Riemenschneider diagram obtained by the removing the last point.
\item Repeatedly blowing down the $-1$ in linear graphs of the form \[ (-a_1,\dots,-a_m,-1,-b_n,\dots,-b_1)\] eventually takes us to $(-2,-1,-2)$, or even further to $(-1,-1)$ or $(0)$.
\item Similarly, blowing up next to the $-1$ gives the linear graph \[(-a_1,\dots,-(a_m+1),-1,-2,-b_n,\dots,-b_1)\] or \[(-a_1,\dots,-a_m,-2,-1,-(b_n+1),\dots,-b_1)\], which are both pairs of complementary legs connected by a $-1$, described by Riemenschneider diagrams that are expansions of the initial one by one dot. 
\end{enumerate}
\end{prop}

\section{Growing Complementary Legs on Lisca's Graphs}
\label{sec:analysis}

The idea for this work comes from studying the lattice embeddings of linear graphs and other trees that are known to bound rational homology 4-balls. Consider for example Lisca's classification of connected linear graphs that bound rational homology 4-balls \cite{liscasingle07}, in the most convenient form for us described by Lecuona in \cite[Remark 3.2]{lecuonamontesinos}. Every family of embeddable graphs can be obtained from the basic graphs $(-2,-2,-2)$, $(-2,-2,-3,-4)$, $(-3,-2,-2,-3)$ and $(-3,-2,-3,-3,-3)$ by repeated application of two types of moves, one of which is the following: choose a basis vector $e$ hitting exactly two vertices $v$ and $w$, where $w$ is \textbf{final} (Lisca's word for leaf, that is a vertex of degree $1$  \cite[p. 6]{liscasingle07}), subtract $1$ from the weight of $v$ and attach a new vertex $u$ of weight $-2$ to $w$. I will show that we can do away with the assumption that $w$ is final and still get 3-manifolds bounding rational homology 4-balls through repeating this operation.

\begin{figure}[H]
\centering
\begin{tikzpicture}
	\begin{pgfonlayer}{nodelayer}
		\node [style={small_black_dot}, label={left:$-2$}, label={right:\color{blue}$e_1-e_2$}] (0) at (-3, 5) {};
		\node [style={small_black_dot}, label={left:$-2$}, label={right:\color{blue}$e_2-e_3$}] (1) at (-3, 3) {};
		\node [style={small_black_dot}, label={left:$-3$}, label={right:\color{blue}$-e_1-e_2-e_4$}] (2) at (-3, 1) {};
		\node [style={small_black_dot}, label={left:$-4$}, label={right:\color{blue}$e_1+e_2+e_3-e_4$}] (3) at (-3, -1) {};
	\end{pgfonlayer}
	\begin{pgfonlayer}{edgelayer}
		\draw [style=new edge style 0] (0) to (1);
		\draw [style=new edge style 0] (1) to (2);
		\draw [style=new edge style 0] (2) to (3);
	\end{pgfonlayer}
\end{tikzpicture}
\caption{Lisca's $(-2,-2,-3,-4)$ graph with embedding.} \label{fig:liscagraph2234}
\end{figure}

\begin{figure}[H]
\centering
\begin{tikzpicture}
	\begin{pgfonlayer}{nodelayer}
		\node [style={small_black_dot}, label={left:$-2$}, label={right:\color{blue}$e_1-e_2$}] (0) at (-3, 5) {};
		\node [style={small_black_dot}, label={left:$-2$}, label={right:\color{blue}$e_2-e_3$}] (1) at (-3, 3) {};
		\node [style={small_black_dot}, label={left:$-3$}, label={right:\color{blue}$-e_1-e_2-e_4$}] (2) at (-3, 1) {};
		\node [style={small_black_dot}, label={left:$-5$}, label={right:\color{blue}$e_1+e_2+e_3-e_4-f_1$}] (3) at (-3, -1) {};
		\node [style={small_black_dot}, label={above:$-2$}, label={left:\color{blue}$e_4-f_1$}] (4) at (-6, 0) {};
	\end{pgfonlayer}
	\begin{pgfonlayer}{edgelayer}
		\draw [style=new edge style 0] (0) to (1);
		\draw [style=new edge style 0] (1) to (2);
		\draw [style=new edge style 0] (2) to (3);
		\draw (2) to (4);
	\end{pgfonlayer}
\end{tikzpicture}
\caption{An expansion of Lisca's $(-2,-2,-3,-4)$ graph with embedding.} \label{fig:liscagraph2234extended1}
\end{figure}

\begin{example}
Consider Figure \ref{fig:liscagraph2234}, showing an embedding of Lisca's $(-2,-2,-3,-4)$ graph into the standard lattice $(\Z \langle e_1, e_2, e_3, e_4 \rangle , -\Id)$. Note that $e_4$ and $e_3$ hit two vertices each. Choose $e_4$. We can now perform the operation described above by choosing $v$ to be the vertex of weight $-4$ and $w$ the vertex of weight $-3$. The result is shown in Figure \ref{fig:liscagraph2234extended1} together with its embedding, which is a kind of ``expansion'' of the embedding in Figure \ref{fig:liscagraph2234}. Our new embedding has two basis vectors hitting exactly two vertices each, namely $e_3$ and $f_1$, whereas $e_4$ now hits three vertices. We may now perform the same operation again on any of these basis vectors, thereby obtaining any graph of the form described in Figure \ref{fig:liscagraph2234extendedgeneral}, with $k=0$. We will show that these graphs do not only have lattice embeddings, but also bound rational homology 4-balls.
\end{example}

\subsection{GOCL and IGOCL moves}

\begin{figure}[H]
\begin{minipage}{\textwidth}
  \begin{subfigure}{\linewidth}
    \centering
    \includegraphics[width=0.7\linewidth]{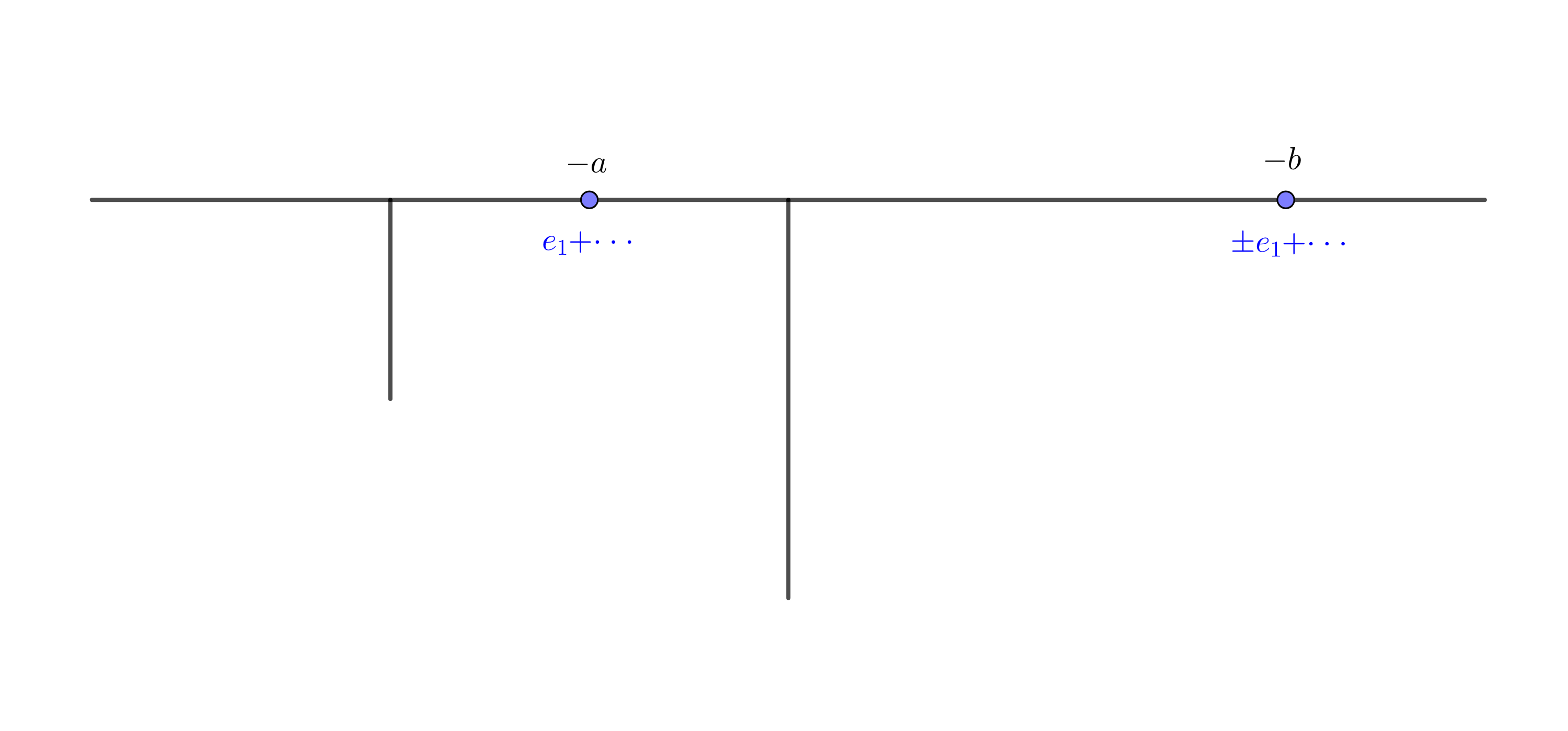}
    \caption{Before GOCL.}
    \label{fig:preGOCL}
  \end{subfigure}
  \begin{subfigure}{\linewidth}
    \centering
    \includegraphics[width=0.7\linewidth]{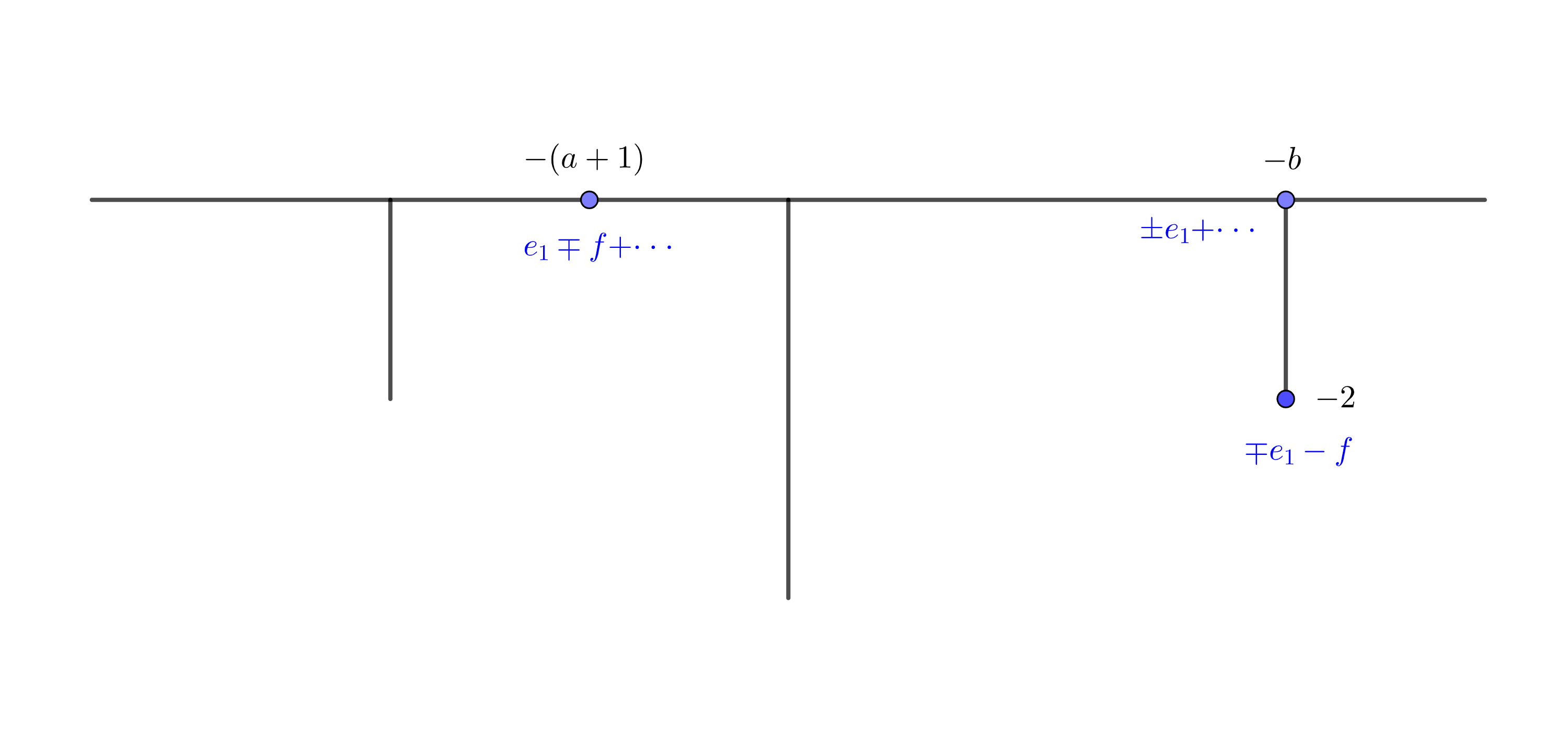}
    \caption{After GOCL.}
    \label{fig:postGOCL}
  \end{subfigure}
\end{minipage}%
\caption{A GOCL move on a graph with a lattice embedding.}
\label{fig:GOCL}
\end{figure}

We will now introduce two moves on forest-shaped plumbing graphs with a lattice embedding. Let $\Gamma=(V,E,W)$ be a weighted negative definite graph with lattice embedding $F: (V,Q_{X_{\Gamma}}) \to (\Z^{|V|},-\Id)$. Assume that there is a basis vector $e$ of $\Z^{|V|}$ hitting exactly two vertices $A$ and $B$ in $\Gamma$, whose images are $v$ and $w$, in any order we prefer. Then a \textbf{GOCL (growth of complementary legs) operation} is constructing an embedded graph $(\Gamma'=(V',E',W'), F')$ by $V'=V\cup C$, $E'=E\cup \{AC\}$ and $u:=F'(C)=-\langle e, v \rangle e - f$, $w':=F'(B)=w-\langle e, v \rangle \langle e, w \rangle f$ and $F'(D)=F(D)$ for all $D \in V-\{B\}$. This move is illustrated by Figure \ref{fig:GOCL}. Note that $\langle e, v \rangle \langle e, w \rangle = \langle f, u \rangle \langle f, w' \rangle$. Thus, the GOCL operation substitutes $e$ by $f$ in the set of basis vectors hitting the graph exactly twice and moreover, the sign difference between the two occurrences of the basis vector is preserved. This operation can therefore be applied repeatedly. If we start with the graph consisting of two vertices of weight $-2$ and no edges, and the embedding $e_1-e_2$ and $e_1+e_2$, then repeated application of GOCL will simply give us two complementary legs.

\begin{figure}[H]
\begin{minipage}{\textwidth}
  \begin{subfigure}{\linewidth}
    \centering
    \includegraphics[width=0.7\linewidth]{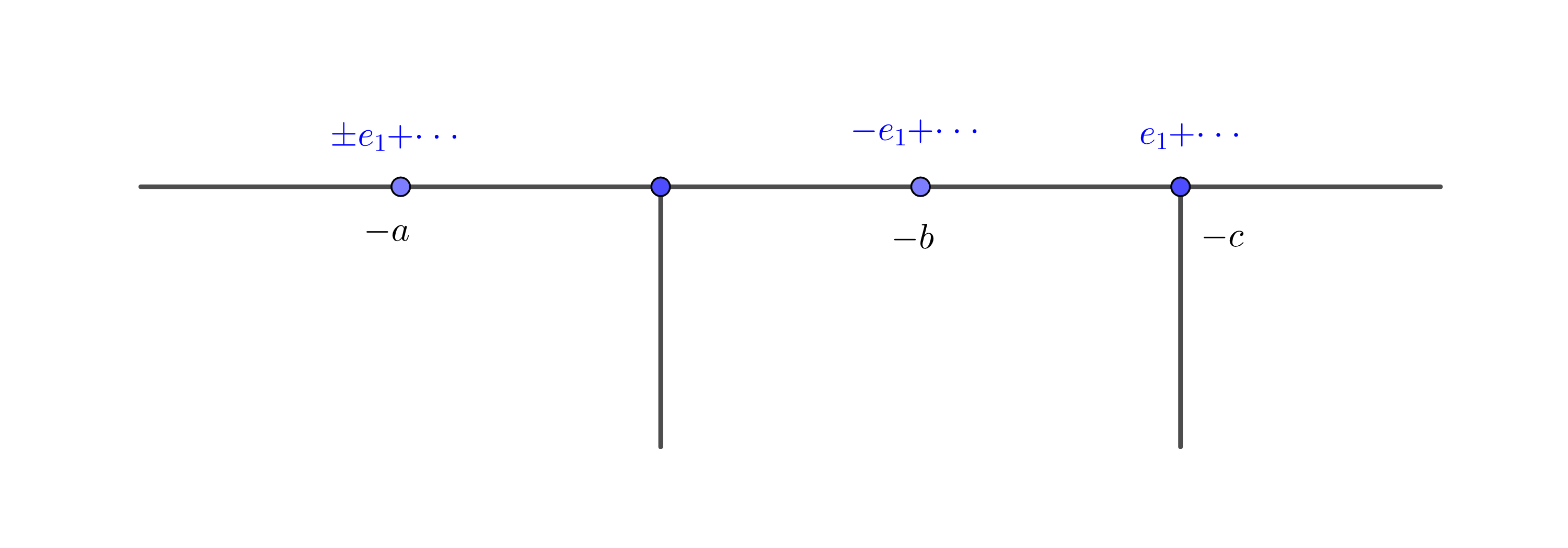}
    \caption{Before IGOCL.}
    \label{fig:preIGOCL}
  \end{subfigure}
  \begin{subfigure}{\linewidth}
    \centering
    \includegraphics[width=0.7\linewidth]{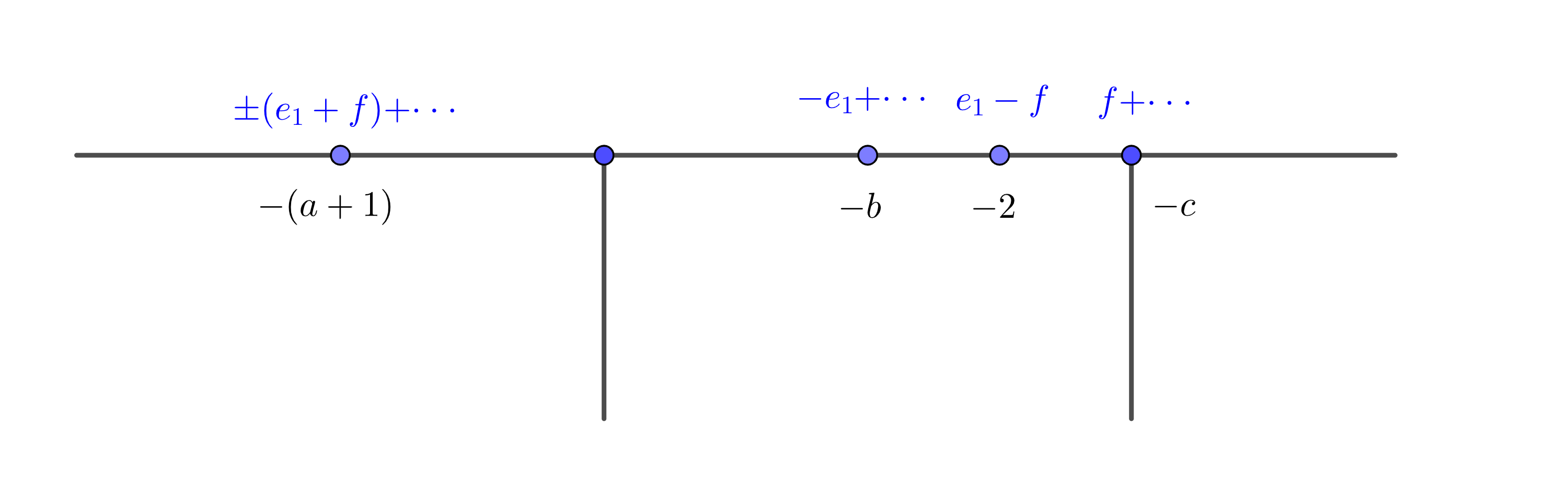}
    \caption{After IGOCL.}
    \label{fig:postIGOCL}
  \end{subfigure}
\end{minipage}%
\caption{An IGOCL move on a graph with a lattice embedding.}
\label{fig:IGOCL}
\end{figure}

The other operation which we will call \textbf{IGOCL (inner growth of complementary legs)} could be described as growing complementary legs from the inside. Suppose a basis vector $e$ hits exactly three vertices $A$, $B$ and $C$ in $\Gamma$, with their images under the lattice embedding $F$ being $u$, $v$ and $w$ respectively. Assume also that $B$ and $C$ are adjacent and that $\langle v, e \rangle \langle w, e \rangle = -1$, that is $e$ hits $v$ and $w$ with opposite signs. Then $\Gamma'=(V', E', W')$ is described by $V'=V\cup \{ D \}$, $E'=(E-\{BC\})\cup \{BD, DC\}$, $F'(D)=-\langle v, e \rangle e + \langle v, e \rangle f$, $F'(C)=w-\langle w, e \rangle e + \langle w, e \rangle f$, $F'(A)=u+\langle u, e \rangle f$ and $F'(X)=F(X)$ for all $X \in V - \{A,C\}$. This operation is illustrated in Figure \ref{fig:IGOCL}. After this operation is performed, we can perform it again on either $e$ or $f$, but the result is essentially the same. What it does is grow a chain of $-2$'s between two vertices and compensate by subtracting from the weight of a different vertex. If we apply the IGOCL operation on a vector hitting a pair of complementary legs three times, we still get a pair of complementary legs, which explains the name.

\subsection{The Complicated Method to Show Theorem \ref{thm:general}} \label{subsec:complicatedmethod}

Now that we have defined the GOCL and IGOCL moves, we want to apply them to Lisca's basic graphs, which are $(-2,-2,-2)$, $(-2,-2,-3,-4)$, $(-3,-2,-2,-3)$, and $(-3,-2,-3,-3,-3)$. We will show for each Lisca graph one by one that the results obtained from repeatedly applying the aforementioned operations always bound rational homology balls. Recall that this can be done for all families using The Simple Method of Proposition \ref{prop:Xhas10n1n2handles}. This subsection explains The Complicated Method to do that, which has the bonus of showing the $\chi$-sliceness of some links including many Montesinos links not treated by Lecuona in \cite{lecuonamontesinos}.

A link in $S^3=\R^3 \cup \{\infty \}$ is called \textbf{strongly invertible} if it is ambient isotopic to one which both is equivariant with respect to the $180^{\circ}$ rotation around the $x$-axis and fulfils the property that each component intersects the $x$-axis in exactly two points \cite{montesinoscoverings}. Now, recall Notation \ref{notation:plumbinggraph}. If $\Gamma$ is a tree, then the link $L_\Gamma$ is strongly invertible. Let $D$ be some Kirby diagram of $X_\Gamma$ such that $L_\Gamma$ is in the equivariant position. For example, Figure \ref{fig:2234involution} shows a possible diagram $D$ for the plumbing graph in Figure \ref{fig:liscagraph2234extended1}.

\begin{figure}[H]
\centering
\includegraphics[width=.7\linewidth]{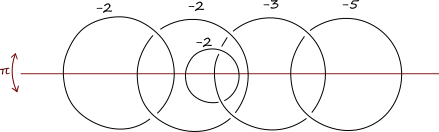}
\caption{Proof that the attaching link of the graph in Figure \ref{fig:liscagraph2234extended1} is strongly invertible.} \label{fig:2234involution}
\end{figure}

Let $\pi_D: X_\Gamma \to X_\Gamma$ be the involution given by extending this $180^\circ$ rotation around the $x$-axis and let $p_D: X_\Gamma \to X_\Gamma/ \pi_D$ be the quotient map when we identify $x \sim \pi_D(x)$. By \cite[Theorem 3]{montesinoscoverings}, $X_\Gamma/ \pi_D +\cong B^4$ and $p$ is a double covering, branched over a surface $S_D \subset B^4$. The surface $S_D$ can be drawn by attaching bands to a disc according to the bottom half of the rotation-equivariant drawing, adding as many half-twists as the weight of the corresponding unknot \cite{montesinoscoverings}. (See Figure \ref{fig:2234arborescent}.) 

\begin{figure}[H]
\centering
\includegraphics[width=.7\linewidth]{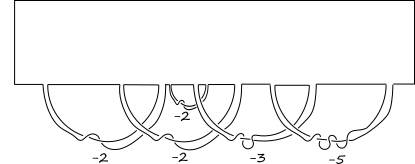}
\caption{$S_D$ for $D$ the diagram in Figure \ref{fig:2234involution} is a disc with five bands attached. It lives inside $B^4$ with the boundary lying on $S^3$, and the interior pushed inside $B^4$.} \label{fig:2234arborescent}
\end{figure}

By $K_D$ we denote the link $K_D=S_D \cap S^3$. Note that $K_D \neq L_\Gamma$. Also note that $K_D = \partial S_D$ and that Figure \ref{fig:2234arborescent} should be understood with the interior pushed into $B^4$. Third, note that $S_D$ and $K_D$ depend on the choice of $D$.

In \cite[Definition 1]{donaldowens}, Donald and Owens define a useful generalisation of sliceness to links. We say that a link in $S^3$ is \textbf{$\chi$-slice} if it bounds a surface of Euler characteristic $1$ without closed components in $B^4$. If the surface has no local maxima, we can call it \textbf{$\chi$-ribbon}. The usefulness of the notion of $\chi$-sliceness lies in the following proposition:

\begin{prop*}[{\cite[Proposition 2.6]{donaldowens}}]
If $L$ is a non-zero-determinant link which bounds a surface $F \subset B^4$ with no closed components and of Euler characteristic $1$, then $\Sigma_2(B^4,F)$ is a rational homology ball.
\end{prop*}

We now state the proposition at the heart of the complicated method to prove Theorem \ref{thm:general}:

\begin{prop}[The Complicated Method]\label{prop:complicatedmethod}
Let $\Gamma$ be a tree-shaped negative-definite plumbing graph, and $D$ an equivariant Kirby diagram of $X_\Gamma$ as above. If we can equivariantly add $n$ $2$-handles with unknotted attaching circles to $D$ and obtain a Kirby diagram $D'$ that describes the $3$-manifold $\#^n(S^1\times S^2)$, then $K_D$ is $\chi$-ribbon, and thus $Y_\Gamma \cong \Sigma_2(S^3,K_D)$ bounds a rational homology ball.
\end{prop}

Before we prove this proposition, let us state the following lemma, which is interesting in itself. This lemma follows directly from combining the Smith Conjecture (proven by Waldhausen in \cite{waldhausensmithconjecture}) and \cite[Proposition 5.1]{heddenni2010}.

\begin{lemma}\label{lem:doublecovers}
If $L$ is a link such that $\Sigma_2(S^3,L)\cong \#^n (S^1 \times S^2)$, then $L$ is the unlink of $n+1$ components.
\end{lemma}

\begin{proof}[{Proof of Proposition \ref{prop:complicatedmethod}}]
Since $\partial \Sigma_2(S^3,K_{D'}) \cong \#^n (S^1 \times S^2)$, $K_{D'}$ must, by Lemma \ref{lem:doublecovers}, be the unlink of $n+1$ components. Since $S_{D'}$ is obtained from $S_D$ by attaching bands, the attachment of bands yields a cobordism $(S^3 \times I, B)$ from $(S^3, K_D)$ to $(S^3,K_{D'})$ with only index $1$ critical points.  Since $K_{D'}$ is the unlink on $n+1$ components, it bounds $n+1$ discs in $B^4$. Let $S$ be the union of $B$ and these discs. It has no maxima and it retracts onto a graph with $n+1$ vertices and $n$ edges, implying that it has Euler characteristic $1$. This shows that $K_D$ is $\chi$-ribbon, and thus, by \cite[Proposition 2.6]{donaldowens}, $Y_\Gamma \cong \partial \Sigma(B^4,S)$, where $\Sigma(B^4,S)$ is a rational homology ball.
\end{proof}

\subsection{$(-2,-2,-2)$}

This graph has embedding $(e_1-e_2,e_2-e_3,-e_1-e_2)$. The only basis vector hitting twice is $e_1$ whose both occurrences are in final vertices. Thus applying the GOCL operation keeps the graph linear and all such graphs have been shown to be bounding rational homology balls by Lisca. In fact, these graphs describe the lens spaces $L(p^2, pq \pm 1)$, for $p>q>0$ with some orientation \cite[Lemma 9.2]{liscasingle07}.

\subsection{$(-3,-2,-3,-3,-3)$} \label{subsec:32333}

In this subsection, we show the following proposition using The Complicated Method described in Subsection \ref{subsec:complicatedmethod}.

\begin{prop}
Every $3$-manifold described by the graph in Figure \ref{fig:32333general} bounds a rational homology $4$-ball.
\end{prop}

In particular, we will apply Proposition \ref{prop:complicatedmethod} to any graph $\Gamma$ in the family described by Figure \ref{fig:32333general}. We will construct an equivariant (with respect to the $180^\circ$ rotation around the $z$-axis) Kirby diagram $D'$ of $S^1 \times S^2$, with all components unknotted and intersecting the $z$-axis exactly twice, such that removing one of the components gives us a Kirby diagram $D$ of $X_\Gamma$.

The proof will be by Kirby calculus, so in Figure \ref{fig:howtoblowdown} we recall the effect of some blow-ups and blow-downs on Kirby diagrams. Recall that if there are $k$ strands of a link component (counted with sign) in a bunch that we are about to perform a $\pm 1$ blow up around, then the framing of that component increases by $\pm k^2$.

\begin{figure}
     \centering
     \begin{subfigure}[b]{0.3\textwidth}
        	\centering
		\includegraphics[width=\textwidth]{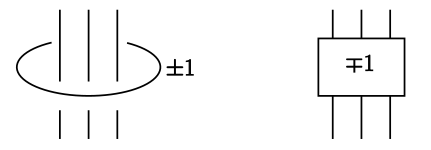}
		\caption{Simple blow-down}
		\label{fig:howtoblowdown1}
     \end{subfigure}
     \hfill
     \begin{subfigure}[b]{0.3\textwidth}
        	\centering
		\includegraphics[width=\textwidth]{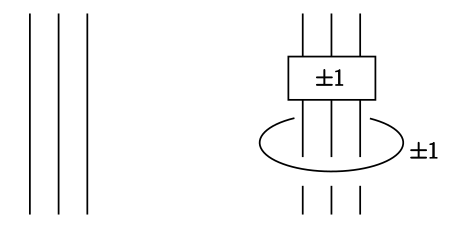}
		\caption{Simple blow-up}
		\label{fig:howtoblowdown2}
     \end{subfigure}
     \hfill
     \begin{subfigure}[b]{0.3\textwidth}
        	\centering
		\includegraphics[width=\textwidth]{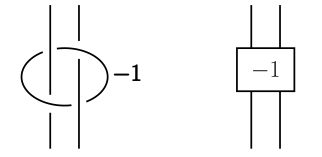}
		\caption{Twisted blow-down}
		\label{fig:howtoblowdowntwistedring}
     \end{subfigure}
        \caption{Useful blow-ups and blow-downs.}
        \label{fig:howtoblowdown}
\end{figure}

\begin{proof}
Now, we start by the chain \[ (-b_{n_1},\dots, -b_1, -(2+k), -1, \underbrace{-2, \dots, -2}_k, -(1+\beta_1), -\beta_2, \dots, -\beta_{n_2}). \] Note that it consists of two Riemenschneider dual chains connected by a $-1$, so by Proposition \ref{prop:latticeembeddingproperties}.4, it blows down to the $(0)$ chain. Since the graph is a tree, the sign of the crossings doesn't matter yet. We will arrange the crossings around the $-(2+k)$-vertex as in Figure \ref{fig:3cycleremovalsimplified}. The chains $(-b_2,\dots, -b_{n_1})$ and $(-2, \dots, -2, -(1+\beta_1), -\beta_2, \dots, -b_{n_2})$, on the other hand, since they are not relevant to the Kirby moves that follow, are represented by tiny squares that freely move on their respective components.

\begin{figure}
     \centering
     \begin{subfigure}[b]{0.3\textwidth}
        	\centering
		\includegraphics[width=\textwidth]{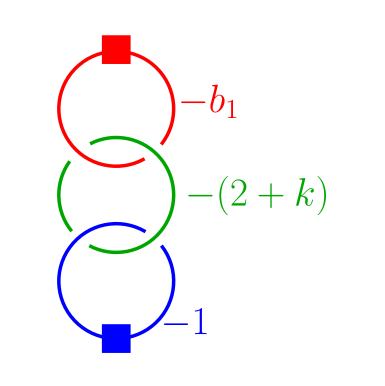}
		\caption{Start}
		\label{fig:3cycleremovalsimplified}
     \end{subfigure}
     \hfill
     \begin{subfigure}[b]{0.3\textwidth}
        	\centering
		\includegraphics[width=\textwidth]{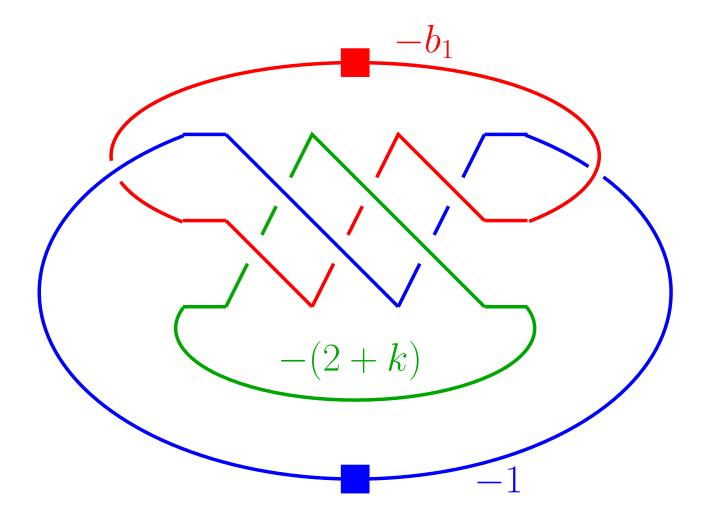}
		\caption{Isotope}
		\label{fig:3cycleremovalblowndown}
     \end{subfigure}
     \hfill
     \begin{subfigure}[b]{0.3\textwidth}
        	\centering
		\includegraphics[width=\textwidth]{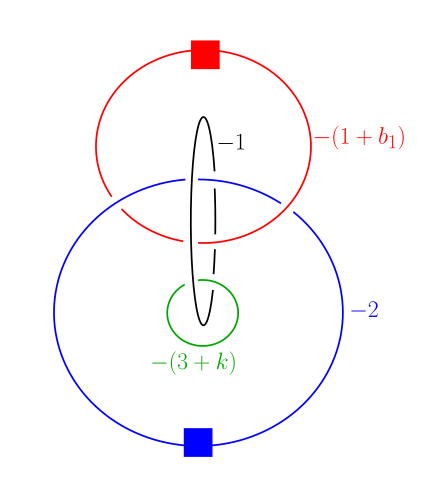}
		\caption{Apply Figure \ref{fig:howtoblowdown1}}
		\label{fig:3cycleremovalblownup}
     \end{subfigure}
        \caption{Creating a $3$-cycle.}
        \label{fig:3cycleremoval}
\end{figure}

Now we apply the Kirby calculus of Figure \ref{fig:3cycleremoval}. The diagram in Figure \ref{fig:3cycleremovalblownup} can no longer be described by a plumbing graph where every vertex is a link component. In fact, the red, blue, and black components form a triple Hopf link. Let us perform a blow-up at the clasp between the blue and the black components. This is a negative clasp, so we need to perform the twisted blow-up of Figure \ref{fig:howtoblowdowntwistedring}. We obtain Figure \ref{fig:32333withlegs}.

\begin{figure}
\centering
\includegraphics[width=.5\linewidth]{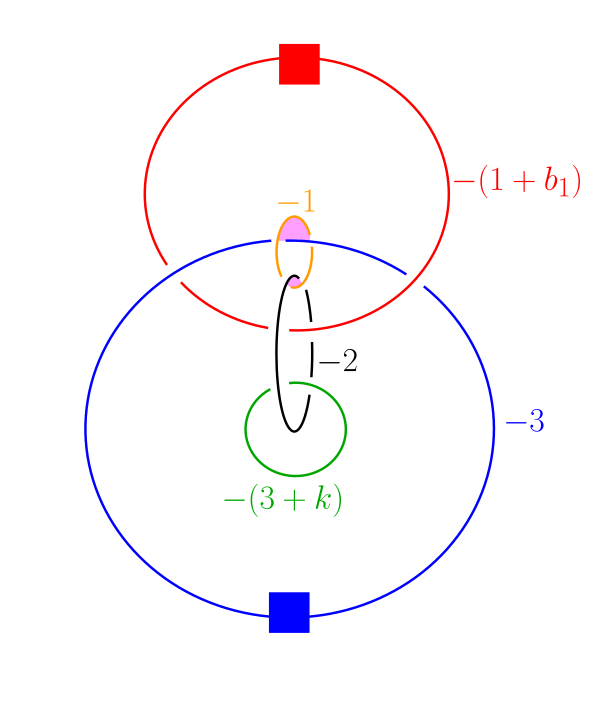}
\caption{}
\label{fig:32333withlegs}
\end{figure}

Now note that removing the amber $-1$-knot in Figure \ref{fig:32333withlegs} yields the Kirby diagram of Figure \ref{fig:32333general} with $(a_1, \dots, a_{m_1})=(2)$. To obtain more general tuples $(a_1,\dots, a_{m_1})$, we will have to repeatedly blow up the clasps on the $-1$-weighted component. At the moment, both of these clasps in Figure \ref{fig:32333withlegs} (in magenta) are negative and can be blown up using Figure \ref{fig:howtoblowdowntwistedring}, an operation that substitutes a clasp by a $-1$-weighted ring with two negative clasps on either side. Thus, repeated blow-ups will give us a figure like Figure \ref{fig:32333withlegs}, but with the amber ring potentially substituted by a longer chain. In any case, the link in this figure is strongly invertible, and removing the $-1$-weighted component yields the Kirby diagram of Figure~\ref{fig:32333general}.
\end{proof}

\subsection{$(-3,-2,-2,-3)$} \label{subsec:3223}

In this subsection, we show the following proposition:

\begin{prop}
Every $3$-manifold described by the graph in Figure \ref{fig:3223general} bounds a rational homology $4$-ball.
\end{prop}

This is done by applying Proposition \ref{prop:complicatedmethod} to graphs $\Gamma$ described by Figure \ref{fig:3223general}. In particular, for each member of this family, we will construct a Kirby diagram $D'$ that 1) describes the $3$-manifold $(S^1 \times S^2)\#(S^1 \times S^2)$, 2) is equivariant with respect to the $180^\circ$ rotation around the $z$-axis, 3) only consists of unknotted components that intersect the $z$-axis exactly twice, and 4) removing two of the components yields a Kirby diagram of $X_\Gamma$.

\begin{proof}
One Kirby diagram of $(S^1 \times S^2)\#(S^1 \times S^2)$ is a $(0,0)$-framed unlink of two components. By Proposition \ref{prop:latticeembeddingproperties}.4, one of the unknots with framing $0$ blows up to the chain \[ (- \beta_{m_2}, \dots, -\beta_2, 1-\beta_1, 1-b_1, -b_2, \dots, -b_{m_1}), \] which can be seen by noting that the above chain is one blow-up away from \[ (- \beta_{m_2}, \dots, -\beta_1, -1, -b_1, \dots, -b_{m_1}). \] Our first move will be to link the other unknot with framing $0$ to the component with framing $1-\beta_1$ using a blow-up, thus obtaining Figure \ref{fig:3223-1}. In this figure, the chains $(-b_2,\dots, -b_{m_1})$ and $(-\beta_2, \dots, \beta_{m_2})$ are represented by tiny squares that freely move on their respective components and if there are two on the same one, they could even pass through each other. Figure \ref{fig:3223-2} is obtained from Figure \ref{fig:3223-1} by a simple isotopy. Note that the purple $-1$-weighted component is linked with the black and the blue ones with negative clasps. We may thus use Figure \ref{fig:howtoblowdowntwistedring} to blow it up into an arbitrary chain of negative clasps as in Figure \ref{fig:3223-3}.

\begin{figure}
     \centering
     \begin{subfigure}[b]{0.25\textwidth}
        	\centering
		\includegraphics[width=\textwidth]{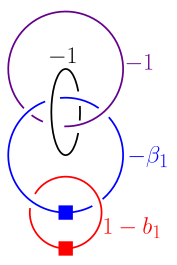}
		\caption{}
		\label{fig:3223-1}
     \end{subfigure}
     \hfill
     \begin{subfigure}[b]{0.3\textwidth}
        	\centering
		\includegraphics[width=\textwidth]{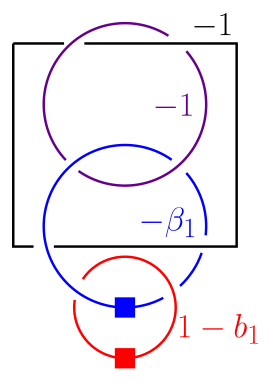}
		\caption{}
		\label{fig:3223-2}
     \end{subfigure}
     \hfill
     \begin{subfigure}[b]{0.4\textwidth}
        	\centering
		\includegraphics[width=\textwidth]{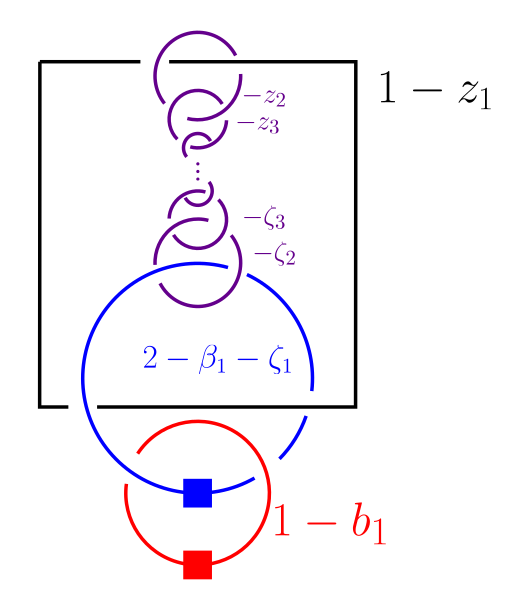}
		\caption{}
		\label{fig:3223-3}
     \end{subfigure}
        \caption{Creating a $\Z_2$-equivariant Kirby diagram of the Figure \ref{fig:3223general} with two extra $2$-handles.}
        \label{fig:3223start}
\end{figure}

Now, zoom into the lower part of Figure \ref{fig:3223-3} and note that it looks like Figure \ref{fig:3braidtwist-1}. It is isotopic to Figure \ref{fig:3braidtwist-2}, which clearly blows up to Figure \ref{fig:3braidtwist-3}. This shows that Figure \ref{fig:3223-3} blows up to Figure \ref{fig:3223-4}. Applying an isotopy of the link gives Figure \ref{fig:3223-5}. The green and the black components are now linked positively. The $-1$-blowup that gets rid of this linking introduces a new component that links to the green and the black components with a negative and positive clasp respectively. Repeated blow-ups thus give us a chain with all clasps negative except the lowest one. We conclude that Figure \ref{fig:3223-6} is a $\Z/2\Z$-equivariant blow-up of Figure \ref{fig:3223-5} and hence of the $(0,0)$ surgery on the $2$-component unlink, but we may also note that removing the two $-1$-weighted components in the ``middle'' of the purple and turquoise chains gives us a tree-shaped plumbing, namely the one in Figure \ref{fig:3223general}.
\end{proof}

\begin{figure}
     \centering
     \begin{subfigure}[b]{0.2\textwidth}
        	\centering
		\includegraphics[width=\textwidth]{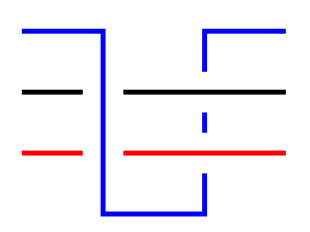}
		\caption{}
		\label{fig:3braidtwist-1}
     \end{subfigure}
     \hfill
     \begin{subfigure}[b]{0.5\textwidth}
        	\centering
		\includegraphics[width=\textwidth]{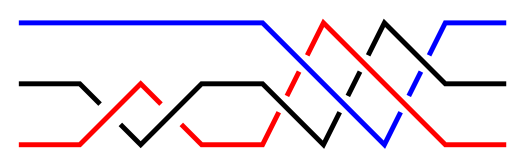}
		\caption{}
		\label{fig:3braidtwist-2}
     \end{subfigure}
     \hfill
     \begin{subfigure}[b]{0.2\textwidth}
        	\centering
		\includegraphics[width=\textwidth]{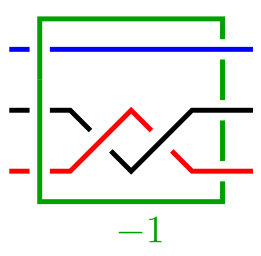}
		\caption{}
		\label{fig:3braidtwist-3}
     \end{subfigure}
        \caption{Zooming in on a part of Figure \ref{fig:3223-3} and blowing up.}
        \label{fig:3braidtwist}
\end{figure}

\begin{figure}
     \centering
     \begin{subfigure}[b]{0.3\textwidth}
        	\centering
		\includegraphics[width=\textwidth]{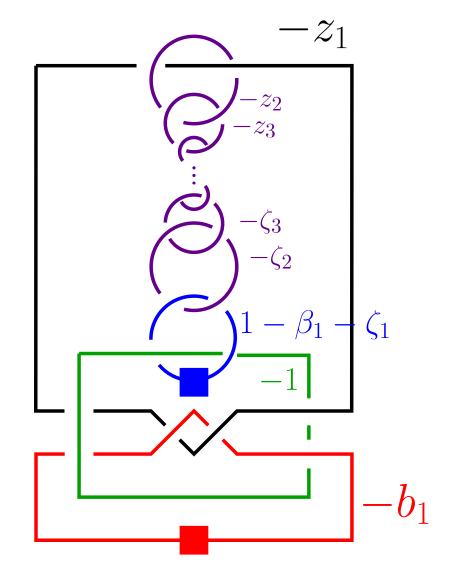}
		\caption{}
		\label{fig:3223-4}
     \end{subfigure}
     \hfill
     \begin{subfigure}[b]{0.3\textwidth}
        	\centering
		\includegraphics[width=\textwidth]{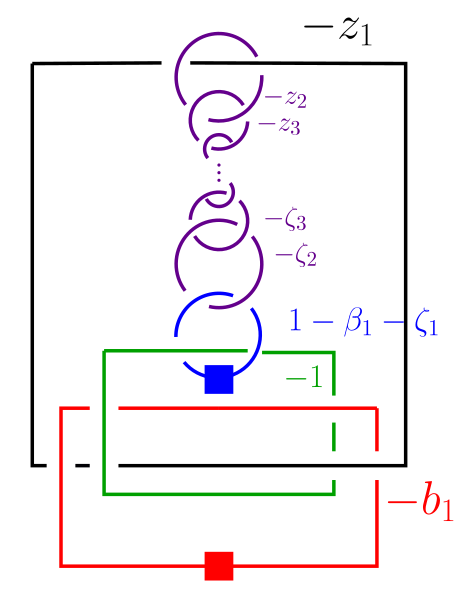}
		\caption{}
		\label{fig:3223-5}
     \end{subfigure}
     \hfill
     \begin{subfigure}[b]{0.3\textwidth}
        	\centering
		\includegraphics[width=\textwidth]{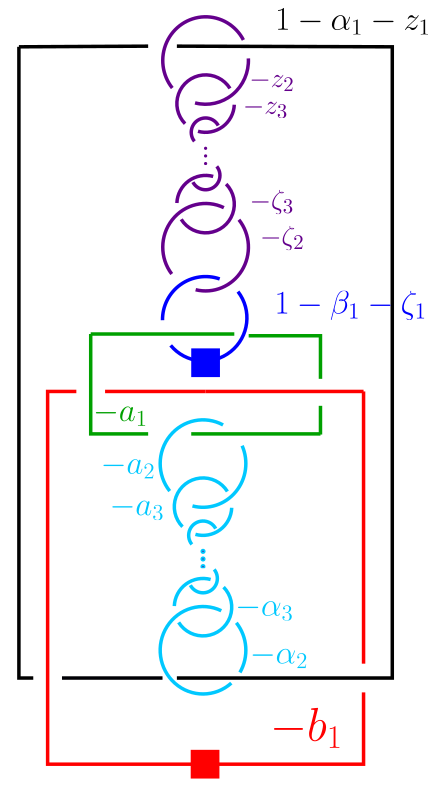}
		\caption{}
		\label{fig:3223-6}
     \end{subfigure}
        \caption{Creating a $\Z/2\Z$-equivariant Kirby diagram of the Figure \ref{fig:3223general} with two extra $2$-handles.}
        \label{fig:3223end}
\end{figure}

\subsection{$(-2,-2,-3,-4)$}

In this subsection, we show the following proposition:

\begin{prop}
Every $3$-manifold described by the graph in Figure \ref{fig:liscagraph2234extendedgeneral} bounds a rational homology $4$-ball.
\end{prop}

This case will be shown in a much simpler way than the one used in Subsections \ref{subsec:32333} and \ref{subsec:3223}. By Proposition \ref{prop:Xhas10n1n2handles}, it is enough to show that performing two integral surgeries on the graphs in Figure \ref{fig:liscagraph2234extendedgeneral} gives us $(S^1 \times S^2) \# (S^1 \times S^2)$.

\begin{proof}
Figure \ref{fig:2234blowdowns1} shows the ``spine'' of Figure \ref{fig:liscagraph2234extendedgeneral} as the orange, magenta, green, red, teal, and violet rings. Note that the complementary leg pairs are reduced to dark brown and olive green squares here. The weights of the rings with squares might be affected by the growth of the complementary legs. The blue ring with weight $-1$ is the first surgery we perform. This first choice of surgery is suggested to us by \cite[Figures 17(5) and 17(6)]{bakerbucklecuona}, which provide equivariant surgeries on \textit{linear} expansions of the graph $(-2,-2,-3,-4)$ that yield $(S^1 \times S^2)$.

Blowing down the $-1$-weighted loop gives us Figure \ref{fig:2234blowdowns2}. Now the red loop has weight $-1$ and we may blow it down. We continue blowing down the chain of $-2$s that follows until we reach Figure \ref{fig:2234blowdowns3}. We blow down the orange loop and obtain Figure \ref{fig:2234blowdowns4}.

Everything would be well now if only the little squares were not there, as we would have been able to blow down the figure to one $0$-framed unknot. However, because we might have grown a two pairs of complementary legs on this figure, Figure \ref{fig:2234blowdowns4} might not actually represent a Kirby diagram with any $-1$-framed loops to blow down.

\begin{figure}
     \centering
     \begin{subfigure}[b]{0.9\textwidth}
        	\centering
		\includegraphics[width=\textwidth]{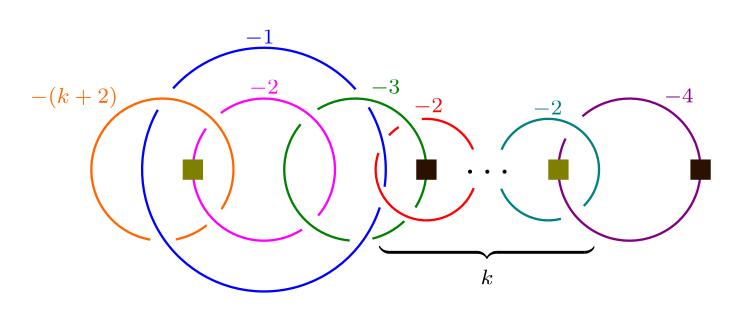}
		\caption{}
		\label{fig:2234blowdowns1}
     \end{subfigure}
     \\
     \begin{subfigure}[b]{0.9\textwidth}
        	\centering
		\includegraphics[width=\textwidth]{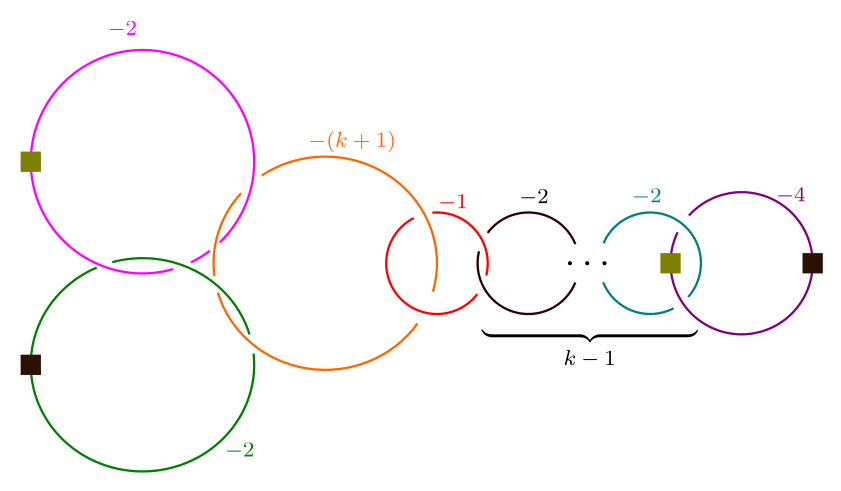}
		\caption{}
		\label{fig:2234blowdowns2}
     \end{subfigure}
     \\
     \begin{minipage}{.4\textwidth}
     \begin{subfigure}[b]{\linewidth}
        	\centering
		\includegraphics[width=\linewidth]{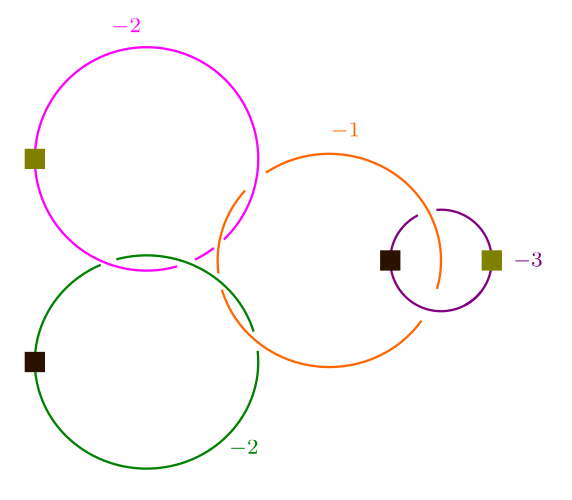}
		\caption{}
		\label{fig:2234blowdowns3}
     \end{subfigure}
     \end{minipage}
     \begin{minipage}{.33\textwidth}
     	\begin{subfigure}[b]{\linewidth}
        		\centering
			\includegraphics[width=\linewidth]{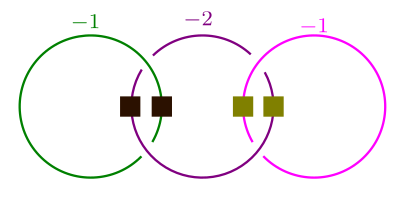}
			\caption{}
			\label{fig:2234blowdowns4}
     	\end{subfigure}
     	\begin{subfigure}[b]{\linewidth}
        		\centering
			\includegraphics[width=\linewidth]{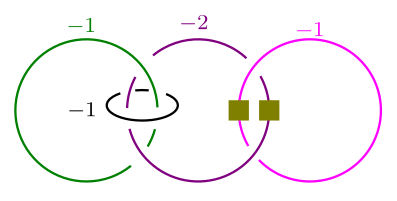}
			\caption{}
			\label{fig:2234blowdowns5}
     	\end{subfigure}
     \end{minipage}
        \caption{The blue component in Subfigure~(a) and the black one in Subfigure~(e) are the two surgeries performed to transform the $3$-manifolds in Figure \ref{fig:liscagraph2234extendedgeneral} into $(S^1 \times S^2) \# (S^1 \times S^2)$.}
        \label{fig:2234blowdowns}
\end{figure}

Instead, we will find another surgery to perform and obtain $(S^1 \times S^2) \# (S^1 \times S^2)$. Here, it's easier to reason backwards. First, note that the $3$-manifold in Figure \ref{fig:2234blowdowns5} is $(S^1 \times S^2) \# (S^1 \times S^2)$. This can be seen by blowing down the black $-1$-loop, which would separate the green loop (now $0$-framed) from the rest. The rest will be a $(-1,-1)$-chain with two complementary legs grown on it, which always blows down to $(0)$. Now, note that repeated blow-ups of Figure \ref{fig:2234blowdowns5} near the black $-1$-framed loop allows us to obtain Figure \ref{fig:2234blowdowns4} with an extra $-1$-framed surgery. This extra $-1$-framed surgery is the one we need to perform in order to obtain $(S^1 \times S^2) \# (S^1 \times S^2)$, completing our proof with The Simple Method.
\end{proof}

While all the links of Figure \ref{fig:2234blowdowns} appear possible to arrange in an equivariant way, the reader should note that we did not equivariantly add the two extra $2$-handles at the same time. Finding an appropriate equivariant diagram $D'$ for using Proposition \ref{prop:complicatedmethod} has proven difficult. It is currently unknown if $K_D$ is $\chi$-slice for any equivariant Kirby diagram $D$ of $X_\Gamma$.


\section{Proof of Theorem \ref{thm:torusknotlist}} \label{sec:torusknots}

In this section, we prove Theorem \ref{thm:torusknotlist} by studying the intersection between the graphs in Figures \ref{fig:3223general}, \ref{fig:32333general}, and \ref{fig:liscagraph2234extendedgeneral}, and plumbing graphs of positive rational surgeries on positive torus knots. First we describe the plumbing graphs of the surgeries on torus knots, and then we go through the intersections with the graphs in Figures  \ref{fig:3223general}, \ref{fig:32333general} and \ref{fig:liscagraph2234extendedgeneral} one by one. The change of order compared to Section \ref{sec:analysis} is because some families obtained from Figure \ref{fig:32333general} are subfamilies of families obtained from Figure \ref{fig:3223general}.

\subsection{Plumbing Graphs of Rational Surgeries on Torus Knots}

\begin{figure}[H]
\centering
\includegraphics[width=.3\linewidth]{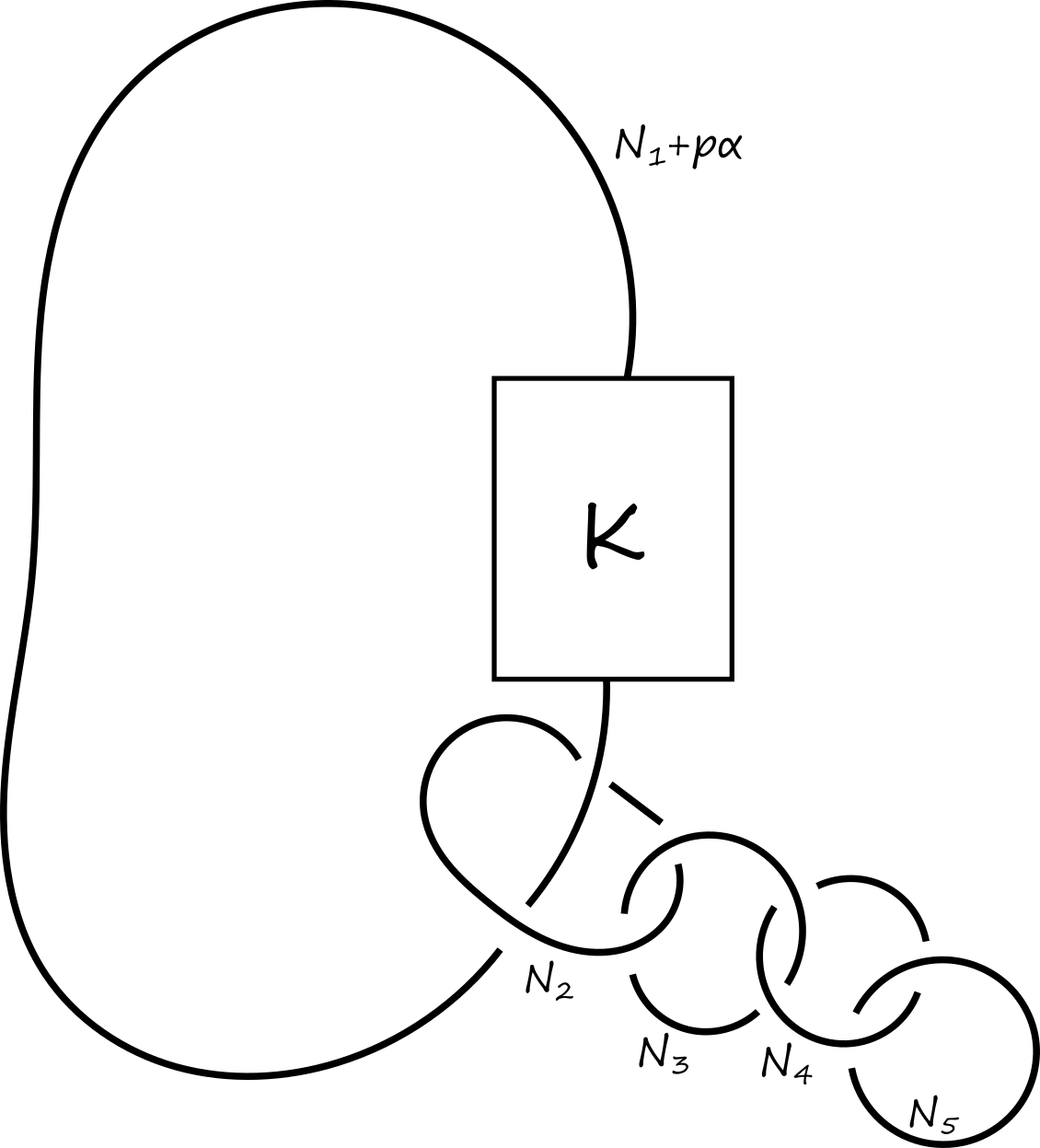}
\caption{A Kirby diagram with boundary $S^3_n(T(p,\alpha))$ where $n=[N_1+p\alpha, N_2, \dots , N_k]^-$ and $K=T(p,\alpha)$, here drawn for $k=5$.}
\label{fig:beforealggeo}
\end{figure}

In order to find the intersection between the plumbing graphs of rational surgeries on torus knots and the graphs obtained from Lisca's graphs by repeated GOCL and IGOCL moves, we need to know what the plumbing graphs of rational surgeries on torus knots look like. Let $n>0$ be a \textit{rational} number. We want to find a plumbing graph for $S^3_n(T(p, \alpha))$. We can write $n=[N_1+p\alpha, N_2, \dots , N_k]^-$ for $N_2,\dots, N_k \geq 2$. The $3$-manifold $S_n(T(p,\alpha))$ bounds the 4-manifold in Figure \ref{fig:beforealggeo}, which is positive-definite if $n>0$. Now, we will use the same technique as in \cite[Section 3]{lokteva2021surgeries} in order to produce a definite plumbing graph. In the process, we need to measure how far we are from being definite, so the following definition is useful.

\begin{dfn}
The \textbf{positive/negative index} of a $4$-manifold is the number of positive/negative eigenvalues of its intersection form.
\end{dfn}

The argument of \cite[Section 3]{lokteva2021surgeries} that the blow-ups decrease the surgery coefficient by a constant still holds to show that $S^3_n(T(p,\alpha))$ bounds the 3-manifold described by the graph in Figure \ref{fig:rationalsurgeryontorusknot}.

\begin{figure}[H]
\centering
\begin{tikzpicture}
	\begin{pgfonlayer}{nodelayer}
		\node [style={small_black_dot}, label={above:$-c_s$}] (1) at (-5, 3) {};
		\node [style={small_black_dot}, label={left:$-1$}] (2) at (-2, 1) {};
		\node [style={small_black_dot}, label={above:$N_1$}] (4) at (1, 1) {};
		\node [style={small_black_dot}, label={above:$N_2$}] (5) at (5, 1) {};
		\node [style=lookingempty] (6) at (8, 1) {$\cdots$};
		\node [style={small_black_dot}, label={above:$N_k$}] (7) at (11, 1) {};
		\node [style={small_black_dot}, label={above:$-c_{s-1}$}] (11) at (-9, 3) {};
		\node [style=lookingempty] (12) at (-12, 3) {$\cdots$};
		\node [style={small_black_dot}, label={above:$-c_2$}] (13) at (-15, 3) {};
		\node [style={small_black_dot}, label={above:$-d_t$}] (14) at (-5, -1) {};
		\node [style=lookingempty] (15) at (-8, -1) {$\cdots$};
		\node [style={small_black_dot}, label={above:$-d_2$}] (16) at (-11, -1) {};
	\end{pgfonlayer}
	\begin{pgfonlayer}{edgelayer}
		\draw [style=new edge style 0] (1) to (2);
		\draw (2) to (4);
		\draw (4) to (5);
		\draw (5) to (6);
		\draw (6) to (7);
		\draw (13) to (12);
		\draw (12) to (11);
		\draw (11) to (1);
		\draw (15) to (14);
		\draw (16) to (15);
		\draw (14) to (2);
	\end{pgfonlayer}
\end{tikzpicture}
\caption{A plumbing graph of $S^3_n(T(p,\alpha))$, where $\alpha>p$. Here $[1,c_2,...,c_s]^-=p/\alpha$, $[d_1,\dots,d_t]^-=\alpha/p$ and $[N_1,\dots,N_k]^-=N=n-p\alpha$. In particular, the pair of fractions $([c_2,\dots,c_s]^-, [d_1,\dots,d_t]^-)$ are complementary. Also, we can write $[(c_2,c_3,\dots,c_s)=(-2^{[d_1-2]},a_1+1, a_2, \dots, a_r)$ so that $[a_1,\dots,a_r]^-$ and $[d_2,\dots, d_t]^-$ are complementary.}
\label{fig:rationalsurgeryontorusknot}
\end{figure}

The positive index of this graph is $k$ by the same logic as in \cite[Section 3]{lokteva2021surgeries}. To obtain a definite graph, we will need the following generalisation of the algorithm in \cite[Figure 2]{lokteva2021surgeries}:

\begin{prop}
Let $\Gamma$ be a tree-shaped plumbing graph containing a chain (a connected linear subgraph with no nodes, that is vertices of degree greater than $2$) $(-\alpha_1,\dots,-\alpha_k)$, as in Figure \ref{fig:changingorientation1}. Let $\Gamma'$ be the graph $\Gamma$ with the chain substituted by the chain $(\beta_1,\dots, \beta_j)$, for $[\alpha_1,\dots, \alpha_k]^-$ and $[\beta_1,\dots, \beta_j]^-$ complementary fractions, and the weight of the vertices adjacent to the chain increased by 1. Then $Y_\Gamma=Y_{\Gamma'}$. Moreover, $b^2_+(X_{\Gamma'})=b^2_+(X_\Gamma)+j$ and $b^2_-(X_{\Gamma'})=b^2_-(X_\Gamma)-k$.
\label{prop:changingorientation}
\end{prop}

\begin{figure}[H]
\begin{minipage}{\textwidth}
  \begin{subfigure}{\linewidth}
    \centering
    \begin{tikzpicture}
	\begin{pgfonlayer}{nodelayer}
		\node [style={small_black_dot}, label={above:$a$}] (0) at (-6, 2) {};
		\node [style={small_black_dot}, label={above:$-\alpha_1$}] (1) at (-3, 2) {};
		\node [style={small_black_dot}, label={above:$-\alpha_k$}] (2) at (9, 2) {};
		\node [style={small_black_dot}, label={above:$b$}] (3) at (13, 2) {};
		\node [style=lookingempty] (5) at (-8, 4) {};
		\node [style=lookingempty] (6) at (-8, 3) {};
		\node [style=lookingempty] (7) at (-8, 2) {};
		\node [style=lookingempty] (9) at (-8, 0) {};
		\node [style=lookingempty] (10) at (15, 4) {};
		\node [style=lookingempty] (11) at (15, 3) {};
		\node [style=lookingempty] (12) at (15, 2) {};
		\node [style=lookingempty] (14) at (15, 0) {};
		\node [style={small_black_dot}, label={above:$-\alpha_2$}] (15) at (0, 2) {};
		\node [style={small_black_dot}, label={above:$-\alpha_{k-1}$}] (16) at (6, 2) {};
		\node [style=lookingempty] (17) at (3, 2) {$\dots$};
		\node [style=lookingempty] (18) at (15, 1) {};
		\node [style=lookingempty] (19) at (-8, 1) {};
	\end{pgfonlayer}
	\begin{pgfonlayer}{edgelayer}
		\draw (0) to (1);
		\draw (2) to (3);
		\draw (0) to (5);
		\draw (0) to (6);
		\draw (0) to (7);
		\draw (0) to (9);
		\draw (3) to (10);
		\draw (3) to (11);
		\draw (3) to (12);
		\draw (3) to (14);
		\draw (1) to (15);
		\draw (15) to (17);
		\draw (17) to (16);
		\draw (16) to (2);
		\draw (19) to (0);
		\draw (3) to (18);
	\end{pgfonlayer}
\end{tikzpicture}
    \caption{Changing a negative chain...}
    \label{fig:changingorientation1}
  \end{subfigure}
  \begin{subfigure}{\linewidth}
    \centering
    \begin{tikzpicture}
	\begin{pgfonlayer}{nodelayer}
		\node [style={small_black_dot}, label={above:$a+1$}] (0) at (-6, 2) {};
		\node [style={small_black_dot}, label={above:$\beta_1$}] (1) at (-3, 2) {};
		\node [style={small_black_dot}, label={above:$\beta_j$}] (2) at (9, 2) {};
		\node [style={small_black_dot}, label={above:$b+1$}] (3) at (13, 2) {};
		\node [style=lookingempty] (5) at (-8, 4) {};
		\node [style=lookingempty] (6) at (-8, 3) {};
		\node [style=lookingempty] (7) at (-8, 2) {};
		\node [style=lookingempty] (9) at (-8, 0) {};
		\node [style=lookingempty] (10) at (15, 4) {};
		\node [style=lookingempty] (11) at (15, 3) {};
		\node [style=lookingempty] (12) at (15, 2) {};
		\node [style=lookingempty] (14) at (15, 0) {};
		\node [style={small_black_dot}, label={above:$\beta_2$}] (15) at (0, 2) {};
		\node [style={small_black_dot}, label={above:$\beta_{j-1}$}] (16) at (6, 2) {};
		\node [style=lookingempty] (17) at (3, 2) {$\dots$};
		\node [style=lookingempty] (18) at (15, 1) {};
		\node [style=lookingempty] (19) at (-8, 1) {};
	\end{pgfonlayer}
	\begin{pgfonlayer}{edgelayer}
		\draw (0) to (1);
		\draw (2) to (3);
		\draw (0) to (5);
		\draw (0) to (6);
		\draw (0) to (7);
		\draw (0) to (9);
		\draw (3) to (10);
		\draw (3) to (11);
		\draw (3) to (12);
		\draw (3) to (14);
		\draw (1) to (15);
		\draw (15) to (17);
		\draw (17) to (16);
		\draw (16) to (2);
		\draw (19) to (0);
		\draw (3) to (18);
	\end{pgfonlayer}
\end{tikzpicture}
    \caption{... to a positive one.}
    \label{fig:changingorientation2}
  \end{subfigure}
\end{minipage}%
\caption{The graphs above bound the same 3-manifold if $[\alpha_1,\dots, \alpha_k]^-$ and $[\beta_1,\dots, \beta_j]^-$ are complementary fractions.}
\label{fig:changingorientation}
\end{figure}

\begin{example*}
Before sketching the proof, we will provide an example of the algorithm that we use to change such a chain. Start with the linear graph $(-2,-4,-2)$. Right now, all vertices have negative weights. We want to introduce a positively weighted vertex. Let us perform a $1$-blow-up. We obtain $(1,-1,-4,-2)$. Now, we blow down the $-1$ and obtain $(2,-3,-2)$. We perform a $1$-blow-up between the $2$ and the $-3$ and obtain $(3,1,-2,-2)$. Blowing up a $1$ again between the last positively weighted vertex and the first negatively weighted one gives us $(3,2,1,-1,-2)$. We blow down the $-1$ to get $(3,2,2,-1)$ and again to obtain $(3,2,3)$. We note that every time we perform a $1$-blow-up, we increase both the positive index and the number of positive vertices by $1$, and every time we perform a $-1$-blowdown, we decrease both the negative index and the number of negative vertices by $1$. Thus, changing these $3$ negative vertices into $3$ positive ones decreased the negative index by $3$ and increased the positive index by $3$.
\end{example*}

\begin{proof}[Proof sketch]
This proposition follows from the fact that blow-ups and blow-downs do not change the boundary $3$-manifold, together with the algorithm of 1) performing a $1$-blow-up at the right of the rightmost chain element greater than $1$, 2) blowing down any $-1$-weighted vertices, and 3) repeating. Following the Riemenschneider diagram, we see that this algorithm gradually substitutes a sequence by its Riemenschneider dual. Blowing up by $1$ increases both the positive index and the number of vertices with positive weight by $1$, and blowing down a $-1$ decreases both the number of vertices with negative weight and the negative index by $1$. Thus, substituting the $k$ negative-weighted vertices by $j$ positive-weighted ones substracts $k$ from the negative index and adds $j$ to the positive index.
\end{proof}

If $N>1$ and thus $N_1\geq 2$, we can use Proposition \ref{prop:changingorientation} to substitute the chain $(N_1,\dots,N_k)$ with its negative Riemenschneider complement $(-M_1,\dots,-M_j)$ and obtain the negative definite graph in Figure \ref{fig:rationalsurgeryontorusknotNpos}. If $0<N<1$, then the sequence $(N_1,\dots, N_k)$ starts with a $1$ possibly followed by some $2$'s that we can blow down before turning the rest of the chain negative. This will once again give us a negative definite graph, namely the one in Figure \ref{fig:rationalsurgeryontorusknotNposlessthan1}.

\begin{figure}
\centering
\begin{tikzpicture}
	\begin{pgfonlayer}{nodelayer}
		\node [style={small_black_dot}, label={above:$-c_s$}] (1) at (-5, 3) {};
		\node [style={small_black_dot}, label={left:$-2$}] (2) at (-2, 1) {};
		\node [style={small_black_dot}, label={above:$-M_1$}] (4) at (1, 1) {};
		\node [style={small_black_dot}, label={above:$-M_2$}] (5) at (5, 1) {};
		\node [style=lookingempty] (6) at (8, 1) {$\cdots$};
		\node [style={small_black_dot}, label={above:$-M_j$}] (7) at (11, 1) {};
		\node [style={small_black_dot}, label={above:$-c_{s-1}$}] (11) at (-9, 3) {};
		\node [style=lookingempty] (12) at (-12, 3) {$\cdots$};
		\node [style={small_black_dot}, label={above:$-c_2$}] (13) at (-15, 3) {};
		\node [style={small_black_dot}, label={above:$-d_t$}] (14) at (-5, -1) {};
		\node [style=lookingempty] (15) at (-8, -1) {$\cdots$};
		\node [style={small_black_dot}, label={above:$-d_2$}] (16) at (-11, -1) {};
	\end{pgfonlayer}
	\begin{pgfonlayer}{edgelayer}
		\draw [style=new edge style 0] (1) to (2);
		\draw (2) to (4);
		\draw (4) to (5);
		\draw (5) to (6);
		\draw (6) to (7);
		\draw (13) to (12);
		\draw (12) to (11);
		\draw (11) to (1);
		\draw (15) to (14);
		\draw (16) to (15);
		\draw (14) to (2);
	\end{pgfonlayer}
\end{tikzpicture}
\caption{A negative definite plumbing graph of $S^3_n(T(p,\alpha))$, where $N=n-p\alpha>1$ and $\alpha>p$. Here $[1,c_2,\dots,c_s]^-=p/\alpha$, $[d_1,\dots,d_t]^-=\alpha/p$ and if $N=n-p\alpha=a/b$ with $a,b \in \Z_{>0}$, then $[M_1,\dots,M_j]^-=\frac{a}{a-b}$.}
\label{fig:rationalsurgeryontorusknotNpos}
\end{figure}

\begin{figure}
\centering
\begin{tikzpicture}
	\begin{pgfonlayer}{nodelayer}
		\node [style={small_black_dot}, label={above:$-c_s$}] (1) at (-5, 3) {};
		\node [style={small_black_dot}, label={left:$-(P_1+1)$}] (2) at (-2, 1) {};
		\node [style={small_black_dot}, label={above:$-P_2$}] (4) at (1, 1) {};
		\node [style={small_black_dot}, label={above:$-P_3$}] (5) at (5, 1) {};
		\node [style=lookingempty] (6) at (8, 1) {$\cdots$};
		\node [style={small_black_dot}, label={above:$-P_j$}] (7) at (11, 1) {};
		\node [style={small_black_dot}, label={above:$-c_{s-1}$}] (11) at (-9, 3) {};
		\node [style=lookingempty] (12) at (-12, 3) {$\cdots$};
		\node [style={small_black_dot}, label={above:$-c_2$}] (13) at (-15, 3) {};
		\node [style={small_black_dot}, label={below:$-d_t$}] (14) at (-5, -1) {};
		\node [style=lookingempty] (15) at (-8, -1) {$\cdots$};
		\node [style={small_black_dot}, label={below:$-d_2$}] (16) at (-11, -1) {};
	\end{pgfonlayer}
	\begin{pgfonlayer}{edgelayer}
		\draw [style=new edge style 0] (1) to (2);
		\draw (2) to (4);
		\draw (4) to (5);
		\draw (5) to (6);
		\draw (6) to (7);
		\draw (13) to (12);
		\draw (12) to (11);
		\draw (11) to (1);
		\draw (15) to (14);
		\draw (16) to (15);
		\draw (14) to (2);
	\end{pgfonlayer}
\end{tikzpicture}
\caption{A negative definite plumbing graph of $S^3_n(T(p,\alpha))$, where $\alpha>p$ and $0<N<1$. Here $[1,c_2,\dots,c_s]^-=p/\alpha$, $[d_1,\dots,d_t]^-=\alpha/p$ and the fraction $[P_1,\dots,P_j]^-$ is complementary to $\frac{1}{1-N}=[N_2,\dots,N_k]^-$. In fact, that means that $N=\frac{1}{[P_1,\dots,P_j]^-}$.}
\label{fig:rationalsurgeryontorusknotNposlessthan1}
\end{figure}

If $N<0$, that is $N_1\leq 0$, then turning the positively-weighted vertices $(N_2, \dots , N_k)$ negative will not be enough to decrease the positive index to $0$. Instead, we will use Proposition \ref{prop:changingorientation} to turn the two other legs of our graph positive, and we obtain the graph in Figure \ref{fig:QsurgeryontorusknotN_1isnonpos}, which has negative index 1. If $N_1=0$, we will perform a 0-absorption (see \cite[Proposition 1.1]{neumann89}) and obtain the positive definite graph in Figure \ref{fig:QsurgeryontorusknotN_1is0}. If $N_1\leq 2$, we use Proposition \ref{prop:changingorientation} to turn it into a chain of $2$s and obtain the graph in Figure \ref{fig:QsurgeryontorusknotN_1isatmostminus2}. If $N_1=-1$, we simply blow it down and obtain Figure \ref{fig:QsurgeryontorusknotN_1isatmostminus2}, but with the length of the chain of $2$s being $0$.

\begin{figure}
\centering
\begin{tikzpicture}
	\begin{pgfonlayer}{nodelayer}
		\node [style={small_black_dot}, label={above:$d_t$}] (1) at (-5, 3) {};
		\node [style={small_black_dot}, label={left:$1$}] (2) at (-2, 1) {};
		\node [style={small_black_dot}, label={above:$N_1$}] (4) at (1, 1) {};
		\node [style={small_black_dot}, label={above:$N_2$}] (5) at (5, 1) {};
		\node [style=lookingempty] (6) at (8, 1) {$\cdots$};
		\node [style={small_black_dot}, label={above:$N_k$}] (7) at (11, 1) {};
		\node [style={small_black_dot}, label={above:$d_{t-1}$}] (11) at (-9, 3) {};
		\node [style=lookingempty] (12) at (-12, 3) {$\cdots$};
		\node [style={small_black_dot}, label={above:$d_1$}] (13) at (-15, 3) {};
		\node [style={small_black_dot}, label={below:$e_r$}] (14) at (-5, -1) {};
		\node [style=lookingempty] (15) at (-8, -1) {$\cdots$};
		\node [style={small_black_dot}, label={below:$e_1$}] (16) at (-11, -1) {};
	\end{pgfonlayer}
	\begin{pgfonlayer}{edgelayer}
		\draw [style=new edge style 0] (1) to (2);
		\draw (2) to (4);
		\draw (4) to (5);
		\draw (5) to (6);
		\draw (6) to (7);
		\draw (13) to (12);
		\draw (12) to (11);
		\draw (11) to (1);
		\draw (15) to (14);
		\draw (16) to (15);
		\draw (14) to (2);
	\end{pgfonlayer}
\end{tikzpicture}
\caption{A plumbing graph of $S^3_n(T(p,\alpha))$, where $\alpha>p$ and $N<0$. Here the negative index is 1, $[d_1,\dots,d_t]^-=\alpha/p$ and $[e_1,\dots,e_r]^-$ is complementary to $[d_2,\dots,d_t]^-$.}
\label{fig:QsurgeryontorusknotN_1isnonpos}
\end{figure}

\begin{figure}
\centering
\begin{tikzpicture}
	\begin{pgfonlayer}{nodelayer}
		\node [style={small_black_dot}, label={above:$d_t$}] (1) at (-5, 3) {};
		\node [style={small_black_dot}, label={left:$N_2+1$}] (2) at (-2, 1) {};
		\node [style={small_black_dot}, label={above:$N_3$}] (4) at (1, 1) {};
		\node [style={small_black_dot}, label={above:$N_4$}] (5) at (5, 1) {};
		\node [style=lookingempty] (6) at (8, 1) {$\cdots$};
		\node [style={small_black_dot}, label={above:$N_k$}] (7) at (11, 1) {};
		\node [style={small_black_dot}, label={above:$d_{t-1}$}] (11) at (-9, 3) {};
		\node [style=lookingempty] (12) at (-12, 3) {$\cdots$};
		\node [style={small_black_dot}, label={above:$d_1$}] (13) at (-15, 3) {};
		\node [style={small_black_dot}, label={below:$e_r$}] (14) at (-5, -1) {};
		\node [style=lookingempty] (15) at (-8, -1) {$\cdots$};
		\node [style={small_black_dot}, label={below:$e_1$}] (16) at (-11, -1) {};
	\end{pgfonlayer}
	\begin{pgfonlayer}{edgelayer}
		\draw [style=new edge style 0] (1) to (2);
		\draw (2) to (4);
		\draw (4) to (5);
		\draw (5) to (6);
		\draw (6) to (7);
		\draw (13) to (12);
		\draw (12) to (11);
		\draw (11) to (1);
		\draw (15) to (14);
		\draw (16) to (15);
		\draw (14) to (2);
	\end{pgfonlayer}
\end{tikzpicture}
\caption{A positive definite plumbing graph of $S^3_n(T(p,\alpha))$, where $\alpha>p$ and $-1<N<0$. Here $[d_1,\dots,d_t]^-=\alpha/p$ and $[e_1,\dots,e_r]^-$ is complementary to $[d_2,\dots,d_t]^-$.}
\label{fig:QsurgeryontorusknotN_1is0}
\end{figure}

\begin{figure}
\centering
\begin{tikzpicture}
	\begin{pgfonlayer}{nodelayer}
		\node [style={small_black_dot}, label={above:$d_t$}] (1) at (-5, 3) {};
		\node [style={small_black_dot}, label={left:$2$}] (2) at (-2, 1) {};
		\node [style={small_black_dot}, label={above:$N_2+1$}] (4) at (6, 1) {};
		\node [style={small_black_dot}, label={above:$2$}] (5) at (4, 1) {};
		\node [style=lookingempty] (6) at (2, 1) {$\cdots$};
		\node [style={small_black_dot}, label={above:$N_k$}] (7) at (12, 1) {};
		\node [style={small_black_dot}, label={above:$d_{t-1}$}] (11) at (-9, 3) {};
		\node [style=lookingempty] (12) at (-12, 3) {$\cdots$};
		\node [style={small_black_dot}, label={above:$d_1$}] (13) at (-15, 3) {};
		\node [style={small_black_dot}, label={below:$e_r$}] (14) at (-5, -1) {};
		\node [style=lookingempty] (15) at (-8, -1) {$\cdots$};
		\node [style={small_black_dot}, label={below:$e_1$}] (16) at (-11, -1) {};
		\node [style={small_black_dot}, label={above:$2$}] (17) at (0, 1) {};
		\node [style=lookingempty] (18) at (10, 1) {$\cdots$};
		\node [style={small_black_dot}, label={above:$N_3$}] (19) at (8, 1) {};
	\end{pgfonlayer}
	\begin{pgfonlayer}{edgelayer}
		\draw [style=new edge style 0] (1) to (2);
		\draw (13) to (12);
		\draw (12) to (11);
		\draw (11) to (1);
		\draw (15) to (14);
		\draw (16) to (15);
		\draw (14) to (2);
		\draw (2) to (17);
		\draw (17) to (6);
		\draw (6) to (5);
		\draw (5) to (4);
		\draw (18) to (7);
		\draw (4) to (19);
		\draw (19) to (18);
	\end{pgfonlayer}
\end{tikzpicture}
\caption{A positive definite plumbing graph of $S^3_n(T(p,\alpha))$, where $\alpha>p$ and $N<-1$. Here $[d_1,\dots,d_t]^-=\alpha/p$ and $[e_1,\dots,e_r]^-$ is complementary to $[d_2,\dots,d_t]^-$, and the tail starts with a chain of $-2$'s of length $-N_1-1$.}
\label{fig:QsurgeryontorusknotN_1isatmostminus2}
\end{figure}

In the graphs of Figures \ref{fig:rationalsurgeryontorusknotNpos} \ref{fig:rationalsurgeryontorusknotNposlessthan1}, \ref{fig:QsurgeryontorusknotN_1is0} and \ref{fig:QsurgeryontorusknotN_1isatmostminus2}, the vertex of degree $3$ is called the node. Removing the node splits the graph into $3$ connected components, of which the top left one is called the \textit{torso}, the bottom left one is called the \textit{leg} and the right one is called the \textit{tail}. This vocabulary is chosen to accord with the vocabulary of \cite{lokteva2021surgeries} on iterated torus knots. We also often talk about the torso, leg and tail collectively as legs. This comes from viewing the graphs as general star-shaped graphs rather than graphs of surgeries on torus knots specifically. (The author recommends looking at a flag of Sicily or Isle of Man for a more precise metaphor.) This vocabulary is generally used by Lecuona, for instance in \cite{lecuonacomplementary} and \cite{lecuonamontesinos}.

We say that two legs of a star-shaped graph are negatively quasi-complementary if either adding one vertex at the end of one leg could make them complementary, and positively quasi-complementary if removing a final vertex from one of the legs could. We say that two legs are complementary if they are either positively or negatively quasi-complementary. Note that the graphs in Figures \ref{fig:rationalsurgeryontorusknotNpos}, \ref{fig:rationalsurgeryontorusknotNposlessthan1}, \ref{fig:QsurgeryontorusknotN_1is0} and \ref{fig:QsurgeryontorusknotN_1isatmostminus2} are exactly the star-shaped graphs with three legs whereof two are quasi-complementary. In the following subsections, we are thus going to look for star-shaped graphs with a pair of quasi-complementary legs among the graphs in Figures \ref{fig:liscagraph2234extendedgeneral}, \ref{fig:3223general} and \ref{fig:32333general}. The following very easy-to-check proposition will come in useful:

\begin{prop} \label{prop:aqcleg}
Suppose $Q/P=[a_1,\dots,a_n]^-$ and $(-a_1,-a_2,\dots,-a_n)$ is either the leg or torso of the plumbing graph of $S_r^3(T(p,\alpha))$, a positive rational surgery on a positive torus knot. (Here $-a_n$ is the weight of the vertex adjacent to the node.) Then $\alpha/p$ is one of the following: \begin{itemize}
\item $\frac{Q}{P}$,
\item $\frac{Q}{Q-P}$,
\item $\frac{(l+1)Q+P}{Q}$ for some $l\geq 0$ or
\item $\frac{(l+2)Q-P}{Q}$ for some $l\geq 0$.
\end{itemize}
\end{prop} Note that if $\GCD(P,Q)=1$, then all of these fractions are reduced. However, if $P=Q-1$, then, $\alpha/p=\frac{Q}{Q-P}=Q$ is a degenerate case that we ignore.

\subsection{$(-3,-2,-2,-3)$}

In this subsection, we prove the following: 

\begin{prop}
For all torus knots $T(p,q)$ in Families \ref{enum:onemodp}, \ref{enum:minusonemodp}, \ref{enum:twomodp}, \ref{enum:minustwomodp}, \ref{enum:3223one} and \ref{enum:3223two} of Theorem \ref{thm:torusknotlist}, there exists an $r \in \Q_+$ such that $S^3_r(T(p,q))$ bounds a rational homology ball.
\end{prop}

This is done by considering the intersections between the graphs in Figures \ref{fig:rationalsurgeryontorusknotNpos}, \ref{fig:rationalsurgeryontorusknotNposlessthan1}, \ref{fig:QsurgeryontorusknotN_1is0} and \ref{fig:QsurgeryontorusknotN_1isatmostminus2} (rational surgeries on torus knots) and the graphs in Figure \ref{fig:3223general} (graphs obtainable from $(-3,-2,-2,-3)$ through GOCL moves). Figure \ref{fig:3223general} is symmetric in the $y$-axis, so it is enough to try two of the vertices for trivalency, say the one with weight $1-\beta_1-\zeta_1$ and the one with weight $-a_1$.

If we want to the vertex with weight $1-\beta_1-\zeta_1$ to be the trivalent vertex in one of the Figures \ref{fig:rationalsurgeryontorusknotNpos}, \ref{fig:rationalsurgeryontorusknotNposlessthan1}, \ref{fig:QsurgeryontorusknotN_1is0} and \ref{fig:QsurgeryontorusknotN_1isatmostminus2} , then $l_1=m_1=1$. Hence $\beta_i=\alpha_j=2$ for all $i$ and $j$. Also, $l_2=1$ or $n_1=1$. Suppose that $(-\beta_{m_2}, \dots, -\beta_{2})$ is one of the quasi-complementary legs. Proposition \ref{prop:aqcleg} would generate that $(p,q)$ belongs to Families \ref{enum:onemodp} and \ref{enum:minusonemodp} in Theorem \ref{thm:torusknotlist}. All of these are possible to produce by setting $l_2=1$, which frees us up to choosing $(\zeta_2, \cdots, \zeta_{n_2})$ completely freely.

\begin{figure}
\centering
\includegraphics[width=.7\linewidth]{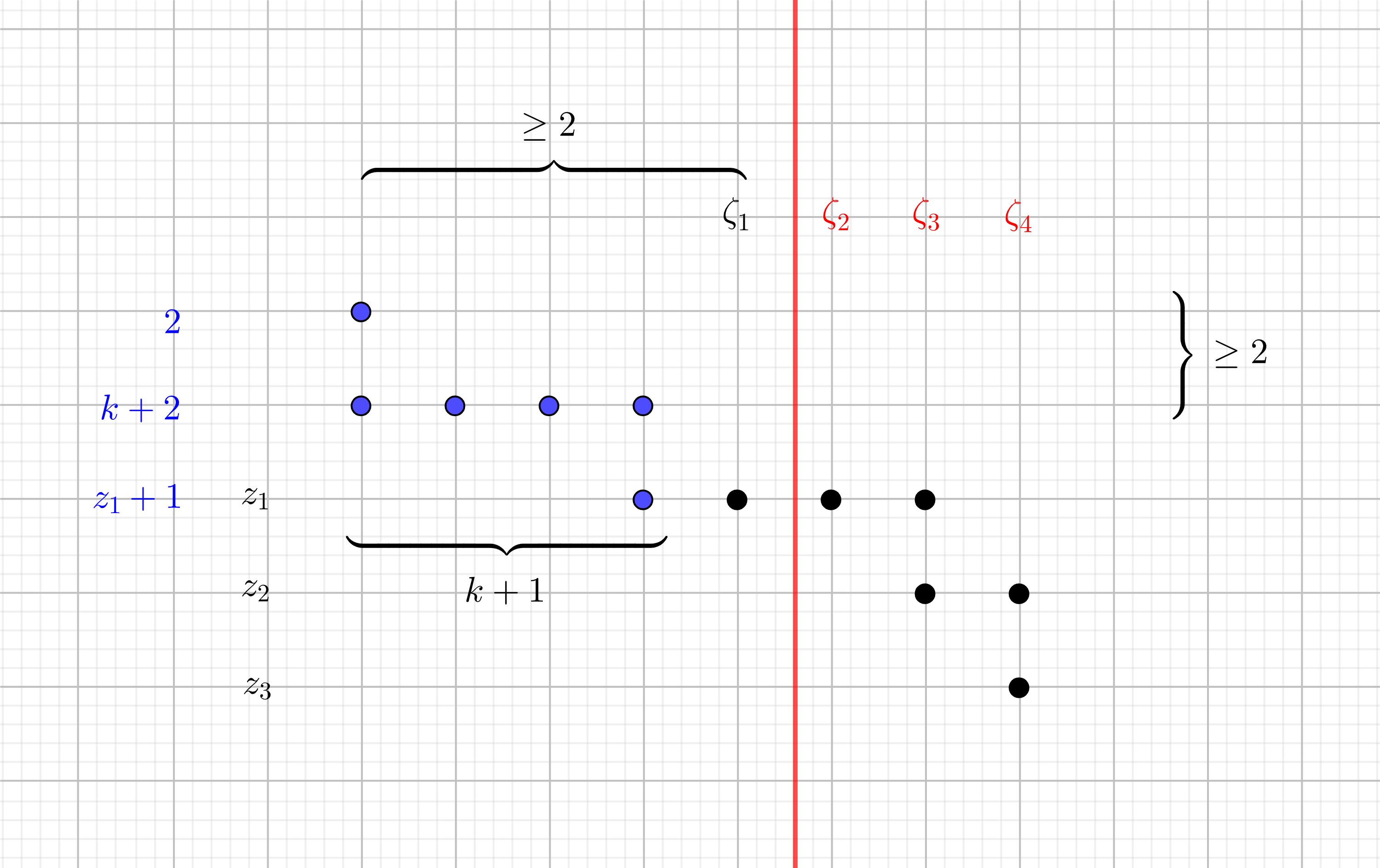}
\caption{It is impossible to choose $k$ so that the length or height difference between the full Riemenshneider diagram and the diagram to the right of the red line is one.}
\label{fig:makegrownlegsqc3223-1}
\end{figure}

Now, we consider what happens if the legs other than $(-\beta_2,\dots, -b_{m_2})$ are quasi-complementary. If $n_1=1$, all $\alpha$s, $\beta$s and $\zeta$s become $-2$, giving us a star-shaped graph with two legs containing nothing but $-2$s, not allowing us out of the families 1, 2, 3 and 4. We consider the case $l_2=1$ instead. We have $a_1=\alpha_1=2$. Let $b_1=k+2$ (so that the leg $(-\beta_2,\dots,-\beta_{m_2})=(-2,\dots,-2)$ has length $k$). We investigate if $(-\zeta_2, \dots, -\zeta_{n_2})$ and $(-2,-(k+2), -z_1-1,-z_2,\dots,-z_{n_1})$ can be quasi-complementary. Consider the diagram in Figure \ref{fig:makegrownlegsqc3223-1}. The black dots represent the Riemenschneider diagram of $(z_1,\dots, z_{n_1})$ and $(\zeta_1, \dots, \zeta_{n_2})$. The blue dots are added in such a way that they together with the black dots form the Riemenschneider diagram of $(-2,-(k+2), -z_1-1,-z_2,\dots,-z_{n_1})$. Call it the BB diagram. The Riemenschneider diagram of $(-\zeta_2, \dots, -\zeta_{n_2})$ is to the right of the red line. Call it the RR diagram. Now we wonder if we can choose the black dots and $k$ in such a way that the BB diagram is just the RR diagram plus one row or column at the end. However, we see that it is impossible to create a difference of one between the length of one leg and the complement of the other leg.

\begin{figure}
\centering
\includegraphics[width=.7\linewidth]{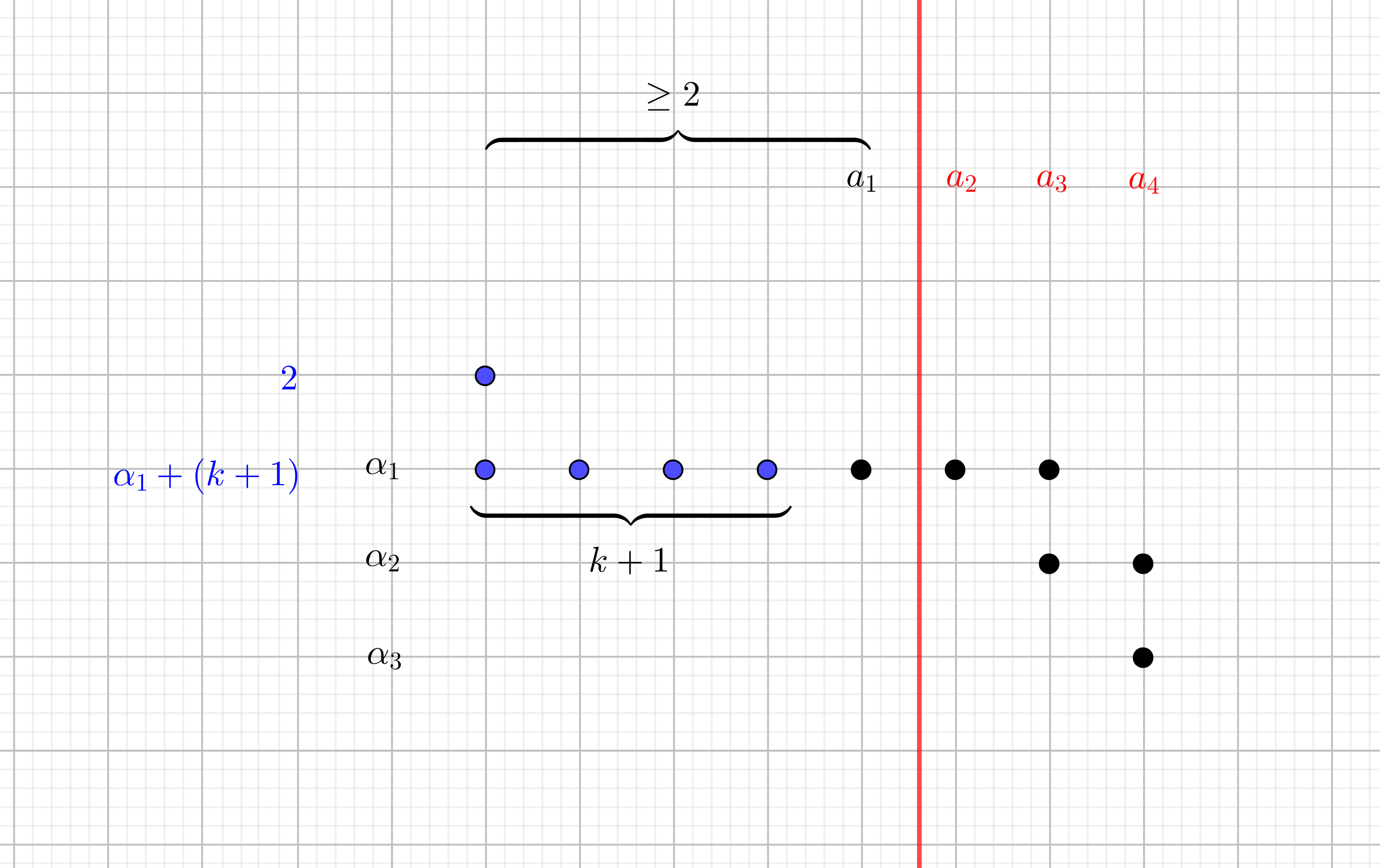}
\caption{Riemenschneider diagram of the quasi-complementary legs $(-a_2,\dots,-a_{l_1})$ and $(-b_1, 1-\alpha_1-z_1, -\alpha_2, \dots, -\alpha_{l_2})$ in Figure \ref{fig:3223general} when $m_1=n_1=m_2=1$.}
\label{fig:makegrownlegsqc3223-2}
\end{figure}

\begin{figure}
\centering
\includegraphics[width=.7\linewidth]{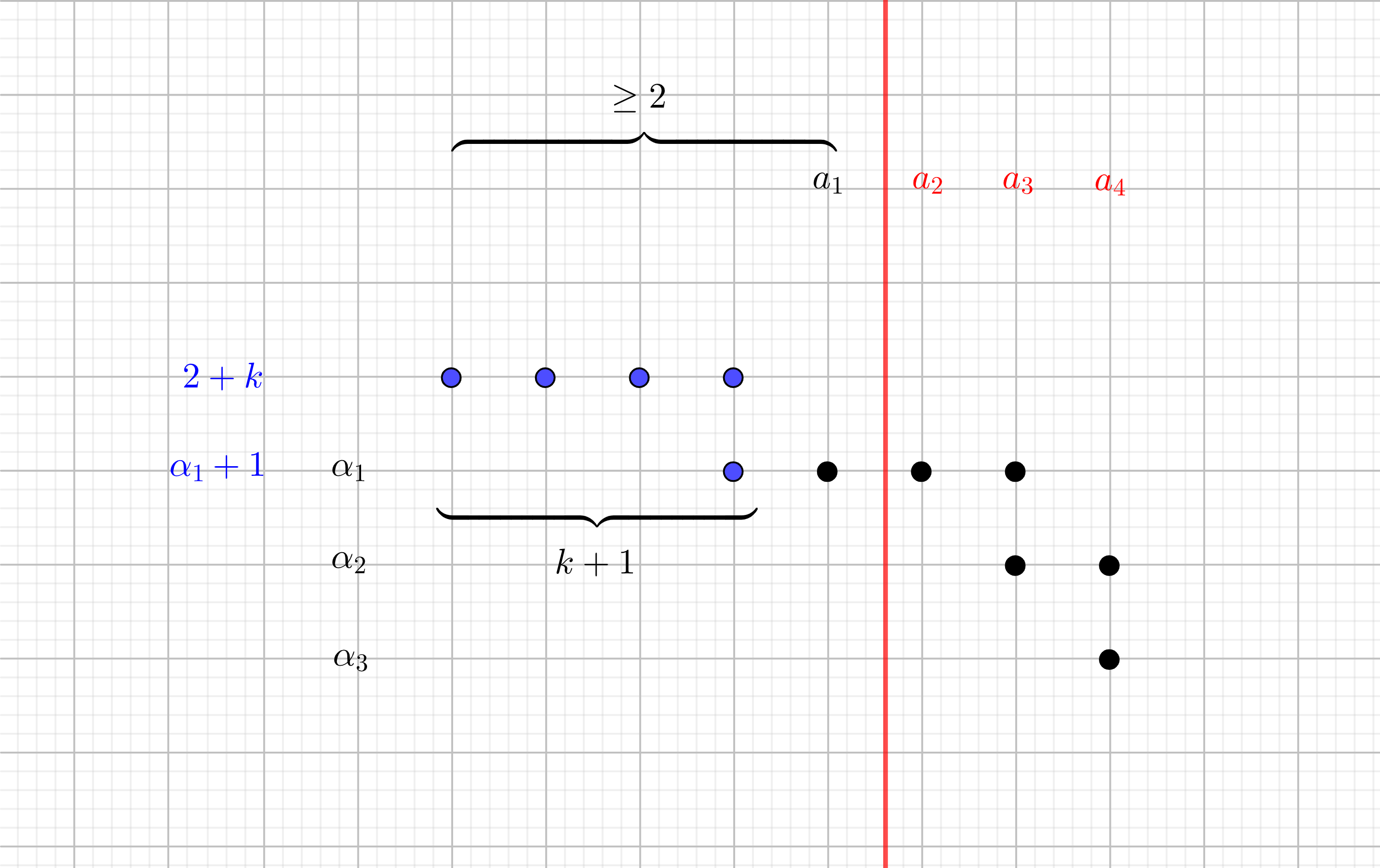}
\caption{Riemenschneider diagram of the quasi-complementary legs $(-a_2,\dots,-a_{l_1})$ and $(-b_1, 1-\alpha_1-z_1, -\alpha_2, \dots, -\alpha_{l_2})$ in Figure \ref{fig:3223general} when $m_1=n_1=n_2=1$.}
\label{fig:makegrownlegsqc3223-3}
\end{figure}

Now, consider the vertex labelled $-a_1$ being trivalent instead. This means that $m_1=1$. Also, either $n_1=1$ or $l_2=1$. First, assume that $n_1=1$. This means that $\beta_1=\cdots=\beta_{m_2}=\zeta_1=\cdots=\zeta_{n_2}=2$. Either $n_2$ or $m_2$ must be $1$. No matter the choice, the left leg becomes $(-2,\dots, -2, -3)$ from the outside. If it is included in a pair of quasi-complementary legs, which we can always ensure since we can choose $(\alpha_2,\dots, \alpha_{l_1})$ freely, we can use Proposition \ref{prop:aqcleg} to get all of families \ref{enum:twomodp} and \ref{enum:minustwomodp}. In the more interesting case (where the leftmost leg is not one of the quasi-complementary ones) $(a_2,\dots,a_{l_1})$ must be quasi-complementary (from the inside) either to $(2,1+k+\alpha_1,\dots,\alpha_{l_2})$ for some $k\geq 0$ (depicted in Figure \ref{fig:makegrownlegsqc3223-2}) or to $(2+k,1+\alpha_1,\dots,\alpha_{l_2})$ for some $k\geq 0$ (depicted in Figure \ref{fig:makegrownlegsqc3223-3}), depending on whether $n_2$ or $m_2$ is equal to $1$.

\begin{figure}
\centering
\includegraphics[width=\linewidth]{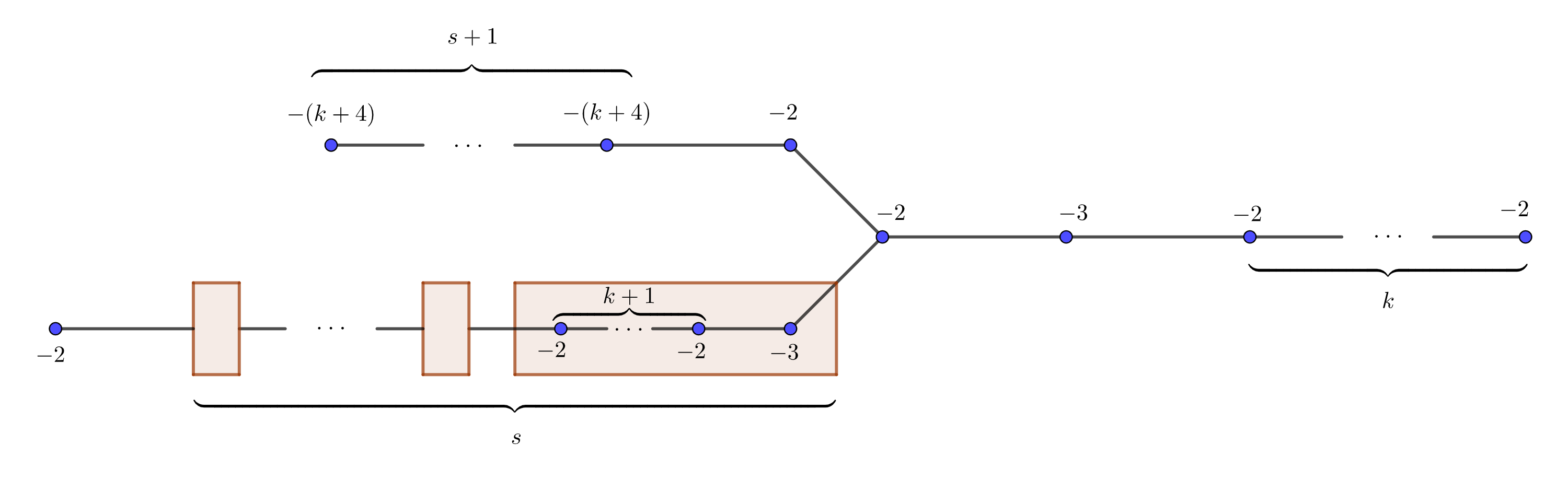}
\caption{A strange two-parameter family of rational surgeries on positive torus knot that bound rational homology 4-balls.}
\label{fig:coolgrowthof3223}
\end{figure}

Let us resolve the first case. Again, in order for $(a_2,\dots,a_{l_1})$ to be quasi-complementary to $(2,1+k+\alpha_1,\dots,\alpha_{l_2})$ for some $k\geq 0$, the BB diagram should be the same as the RR diagram plus an extra row or column at the end. Since the part to the left of the red line has at least two columns, it must be an extra row. One solution would be $(\alpha_1, \dots, \alpha_{l_2})=(3)$. If the black diagram has more than one row, we need $\alpha_2=(k+1)+2+1=k+4$. We can add as many rows as we want this way. We get that $(\alpha_1,\dots,\alpha_{l_2})=(3,(k+4)^{[s]})$ and $(a_1,\dots,a_{l_1})=(2,(3,(2)^{[k+1]})^{[s]},2)$, giving us the graph in Figure \ref{fig:coolgrowthof3223}. This graph is of the shape of Figure \ref{fig:QsurgeryontorusknotN_1isatmostminus2}, so it describes $S^3_n(T(p,\alpha))$ for $\alpha / p= [(k+4)^{[s+1]},2]^-$ and $N=n-p\alpha=[-1,(2)^{[k+1]}]^-=-\frac{2k+3}{k+2}$. This corresponds to family \ref{enum:3223one} in Theorem \ref{thm:torusknotlist}. A different formulation of the result is that $S^3_{p\alpha-\frac{2k+3}{k+2}}(T(p,\alpha))$ bounds a rational homology ball for all $p$ and $\alpha$ described by \[
\begin{pmatrix}
\alpha\\
p
\end{pmatrix} = \begin{pmatrix}
k+4 & -1\\
1 & 0
\end{pmatrix}^{s+1} \begin{pmatrix}
2\\
1
\end{pmatrix}
\] for some $s,k \geq 0$. If we fix $s$, then $\alpha$ becomes a degree $s+1$ polynomial in $k$.

\begin{figure}
\centering
\includegraphics[width=\linewidth]{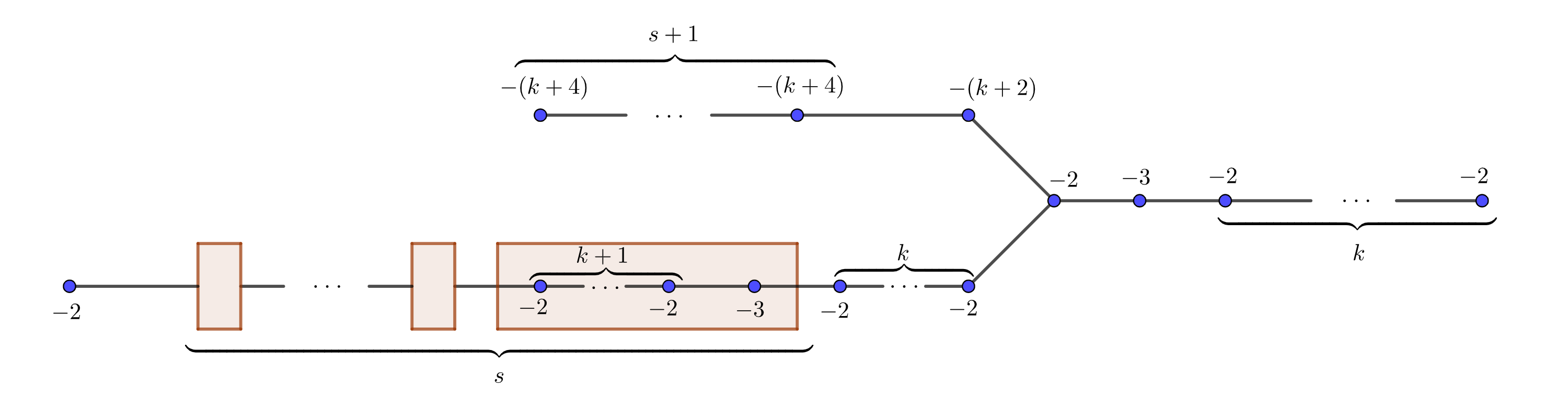}
\caption{A strange two-parameter family of rational surgeries on positive torus knot that bound rational homology 4-balls.}
\label{fig:coolgrowthof3223-2}
\end{figure}

In the second case, that is if $(a_2,\dots,a_{l_1})$ is quasi-complementary to $(2+k,1+\alpha_1,\dots,\alpha_{l_2})$ for some $k\geq 0$, then $(\alpha_1,\dots,\alpha_{l_2})=(k+3,(k+4)^{[s]})$ for some $s\geq 0$. Then the graph becomes as in Figure \ref{fig:coolgrowthof3223-2}. Now $\alpha /p=[(k+4)^{[s+1]}, k+2]^-$, meaning that $S^3_{p\alpha-\frac{2k+3}{k+2}}(T(p,\alpha))$ bounds a rational homology ball for all $p$ and $\alpha$ described by \[
\begin{pmatrix}
\alpha\\
p
\end{pmatrix} = \begin{pmatrix}
k+4 & -1\\
1 & 0
\end{pmatrix}^{s+1} \begin{pmatrix}
k+2\\
1
\end{pmatrix}
\] for some $s,k \geq 0$. This corresponds to Family \ref{enum:3223two} in Theorem \ref{thm:torusknotlist}.

If $l_2=1$ instead of $n_1=1$, then $a_1=\cdots=a_{l_1}=2$. We already know that we can choose surgery coefficients when one of the complementary legs consists of only $-2$s, so we do not need to check that case to formulate Theorem \ref{thm:torusknotlist}. In fact we do not need to check further, as any star-shaped graphs with three legs whereof two are quasi-complementary, the third one consisting only of $-2$s and the node having weight $-2$ is a positive integral surgery on a positive torus knot, which have been classified in \cite{GALL}.

\subsection{$(-3,-2,-3,-3,-3)$}

In this subsection, we prove the following: 

\begin{prop}
For all torus knots $T(p,q)$ in the families \ref{enum:32333cubic} and \ref{enum:32333recursive} of Theorem \ref{thm:torusknotlist}, there exists an $r \in \Q$ such that $S^3_r(T(p,q))$ bounds a rational homology ball.
\end{prop}

This is done by finding the intersections between the graphs in Figures \ref{fig:rationalsurgeryontorusknotNpos}, \ref{fig:rationalsurgeryontorusknotNposlessthan1}, \ref{fig:QsurgeryontorusknotN_1is0} and \ref{fig:QsurgeryontorusknotN_1isatmostminus2} (rational surgeries on torus knots) and the graphs in Figure \ref{fig:32333general} (graphs obtainable from $(-3,-2,-3, -3,-3)$ through GOCL and IGOCL moves).

In Figure \ref{fig:32333general} there are three possibilities for a trivalent vertex. If we choose the vertex of weight $-a_1$, then $m_2=1$ and thus two of the legs are $(-3-k)$ and $(-2,\dots,-2)$. We already know that if one of these is in a quasi-complementary pair, then $(p, \alpha)$ lies in families \ref{enum:onemodp}-\ref{enum:minustwomodp} in Theorem \ref{thm:torusknotlist}, so we get nothing new. Choosing the vertex of weight $-(1+b_1)$ to be trivalent, and noting that we land in families \ref{enum:onemodp}-\ref{enum:minustwomodp} if the left leg $(-3-k,-2)$ is one of the quasi-complementary ones, does however lead us to find that \[ S^3_{p\alpha-\frac{5}{7}}(T(p,\alpha)) \] bounds a rational homology ball for every \[
\begin{pmatrix}
\alpha\\
p
\end{pmatrix} = \begin{pmatrix}
5 & -1\\
1 & 0
\end{pmatrix}^{s+1} \begin{pmatrix}
3\\
1
\end{pmatrix}
\] where $s\geq 0$. This corresponds to family \ref{enum:32333recursive} in Theorem \ref{thm:torusknotlist}. Finally, choosing the vertex of weight $-(1+\alpha_1)$ to be trivalent gives us $m_1=n_1=1$. If the lower leg $(-2,\dots,-2)$ is included in the pair of quasi-complementary legs, we fall into families 1-4 again. We need to investigate when $(-(3+k),-a_1,-(1+b_1))$ can be quasi-complementary to $((-2)^{[b_1-2]},-3,(-2)^{[k]})$.  The Riemenschneider dual of the latter leg is $(-b_1,-(k+2))$, so we need $b_1=a_1$ and $k+2=b_1+1$. Note that we also need $a_1\geq 3$ in order to get a three-legged graph. Let $a=a_1-3$. We get $(-(3+k),-a_1,-(1+b_1))=(-(a+5),-(a+3),-(a+4))$. Our graph is now as in Figure \ref{fig:QsurgeryontorusknotN_1is0}. Thus $\alpha/p=[a+5,a+3,a+4]^-=\frac{a^3+12a^2+45a+51}{a^2+7a+11}$. This correspond to family \ref{enum:32333cubic} in Theorem \ref{thm:torusknotlist}.

The remaining families follow from \cite{lecuonamontesinos} and \cite{GALL}.

\subsection{$(-2,-2,-3,-4)$}

In this subsection, we prove the following: 

\begin{prop}
For all torus knots $T(p,q)$ in the families \ref{enum:2234gall}, \ref{enum:2234one} and \ref{enum:2234two} of Theorem \ref{thm:torusknotlist}, there exists an $r \in \Q$ such that $S^3_r(T(p,q))$ bounds a rational homology ball.
\end{prop}

This is done by determining the intersection between the graphs in Figures \ref{fig:rationalsurgeryontorusknotNpos}, \ref{fig:rationalsurgeryontorusknotNposlessthan1}, \ref{fig:QsurgeryontorusknotN_1is0} and \ref{fig:QsurgeryontorusknotN_1isatmostminus2} and the graphs in Figure \ref{fig:liscagraph2234extendedgeneral}, that is between the graphs of surgeries on torus knots and the graphs of Figure \ref{fig:liscagraph2234extendedgeneral}, which we now know to bound rational homology balls.

To turn Figure \ref{fig:liscagraph2234extendedgeneral} into a star-shaped graph, we will need to keep some of the grown complementary legs to length 1. If we let the vertex of weight $-1-a_1$ be trivalent, then $m_1=1$ and thus $(\beta_1,\dots, \beta_{m_2})$ consists only of $2$'s. If $m_2>1$ then $n_2=1$ and $(a_1,\dots,a_{n_1})=(2,\dots, 2)$. In order to have trivalency of the $-(1+a_1)$ vertex, $\alpha_1\geq 3$ is required. It is easy to check that in this case the only legs that can be quasi-complementary are the $(-b_1,-(2+k))$ one and the $(\underbrace{-2,\dots,-2}_{\alpha_1-2})$ one. They can either be negatively quasi-complementary, made complementary by adding $-3$ at the end of the second leg, in which case $\alpha_1=b_1$ and $k=0$ have to hold, or they can be positively quasi-complementary, made complementary by removing $-(2+k)$ from the first one, in which case $\alpha_1-2=b_1-1$. The first case shows that \[ S^3_{\frac{(2b_1^2-2b_1+1)^2}{2b_1^2-b_1+1}}(T(b_1-1,2b_1-1)) \] bounds a rational homology ball for any $b_1\geq 3$. The second case shows that \[ S^3_{\frac{((k+2)b_1^2-1)^2}{(k+2)b_1^2+b_1-1}}(T(b_1,b_1(k+2)-1)) \] bounds a rational homology ball for all integers $b_1 \geq 2$ and $k \geq 0$. Both of these are subfamilies to families 1 and 2 in Theorem \ref{thm:torusknotlist} that we will show can in fact be fully realised.

We get more interesting families when we let $m_2=1$, because then $(\alpha_1,\dots, \alpha_{n_2})$ can be anything as long as it has something but a $2$ somewhere so that $n_1>1$. We will get graphs of the form in Figure \ref{fig:bandbetagone}. To make the top and right legs quasi-complementary is easy: we need to choose whether they are to be positively or negatively quasi-complementary and which leg needs an extra vertex or a vertex removed to be complementary, and then we just need to choose $(a_2,\dots,a_{n_1})$ that make it happen. We use Proposition \ref{prop:aqcleg} for $Q/P=[2+k,2]^-=\frac{2k+3}{2}$. This corresponds to the entire families 3 and 4 as well as subfamilies of families 1 and 2 in Theorem \ref{thm:torusknotlist}. The top and bottom legs cannot be made quasi-complementary.

\begin{figure}
\centering
\begin{tikzpicture}
	\begin{pgfonlayer}{nodelayer}
		\node [style={small_black_dot}, label={left:$-(2+k)$}] (0) at (-3, 5) {};
		\node [style={small_black_dot}, label={right:$-2$}] (1) at (-3, 3) {};
		\node [style={small_black_dot}, label={left:$-1-a_1$}] (2) at (-3, 1) {};
		\node [style={small_black_dot}, label={below:$-\alpha_1-2$}] (3) at (-3, -7.75) {};
		\node [style={small_black_dot}, label={above:$-a_2$}] (4) at (-1, 1) {};
		\node [style={small_black_dot}, label={above:$-a_3$}] (5) at (1, 1) {};
		\node [style=lookingempty] (6) at (2.5, 1) {$\cdots$};
		\node [style={small_black_dot}, label={above:$-a_{n_1}$}] (7) at (4, 1) {};
		\node [style={small_black_dot}, label={above:$-\alpha_2$}] (8) at (-1, -7.75) {};
		\node [style=lookingempty] (9) at (0.5, -7.75) {$\cdots$};
		\node [style={small_black_dot}, label={above:$-\alpha_{n_2}$}] (10) at (2, -7.75) {};
		\node [style={small_black_dot}, label={right:$-2$}] (18) at (-3, -1) {};
		\node [style=lookingempty] (19) at (-3, -3) {$\vdots$};
		\node [style={small_black_dot}, label={right:$-2$}] (20) at (-3, -5) {};
	\end{pgfonlayer}
	\begin{pgfonlayer}{edgelayer}
		\draw [style=new edge style 0] (0) to (1);
		\draw [style=new edge style 0] (1) to (2);
		\draw (2) to (4);
		\draw (4) to (5);
		\draw (5) to (6);
		\draw (6) to (7);
		\draw (3) to (8);
		\draw (8) to (9);
		\draw (9) to (10);
		\draw (2) to (18);
		\draw (18) to (19);
		\draw (19) to (20);
		\draw (20) to (3);
	\end{pgfonlayer}
\end{tikzpicture}
\caption{Choosing $m_1=m_2=1$ in Figure \ref{fig:liscagraph2234extendedgeneral} gives this star-shaped graph.}
\label{fig:bandbetagone}
\end{figure}

\begin{figure}
\centering
\includegraphics[width=.7\linewidth]{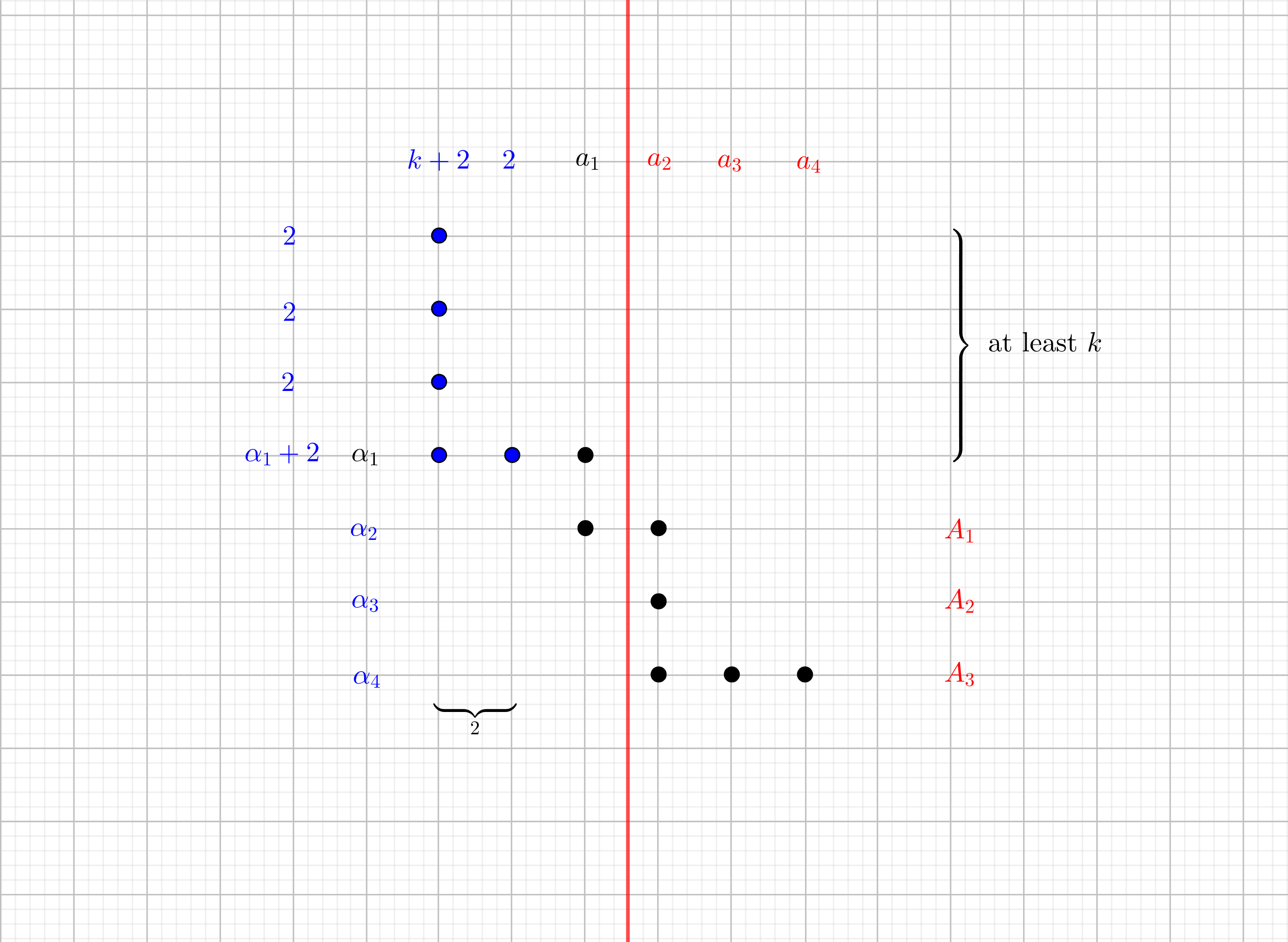}
\caption{}
\label{fig:makegrownlegsqc}
\end{figure}

The most interesting case to consider is whether the right and the bottom legs can be made quasi-complementary. In Figure \ref{fig:makegrownlegsqc}, the black dots show a Riemenschneider diagram of the complementary sequences $(a_1,\dots,a_{n_1})$ and $(\alpha_1,\dots,\alpha_{n_2})$. Adding the blue dots gives us a Riemenschneider diagram for the sequence $(\underbrace{2,\dots,2}_k,\alpha_1+2,\alpha_2, \dots , \alpha_{n_2})$ (with complement $(k+2, 2, a_1, \dots, a_{n_1})$). Considering only the part to the right of the red line gives us a Riemenschneider diagram for $(a_2,\dots, a_{n_1})$ (with a complement $(A_1,\dots,A_{n_3})$). In order for $(\underbrace{2,\dots,2}_k,\alpha_1+2,\alpha_2, \dots , \alpha_{n_2})$ and $(a_2,\dots, a_{n_1})$ to be quasi-complementary, either the picture to the right of the red line and the total picture without the last line, or the total picture and the picture to the right of the the red line with an extra column, must be the same. The sequences $(k+2, 2, a_1, \dots, a_{n_1})$ and $(a_2,\dots, a_{n_1})$ have length difference $2$, removing the second option. The only ways in which $(\underbrace{2,\dots,2}_k,\alpha_1+2,\alpha_2, \dots , \alpha_{n_2})$ and $(A_1,\dots,A_{n_3})$ can have length difference $1$ is if any of the following hold: \begin{enumerate}
\item $k=0$ and $a_1=3$, or
\item $k=1$ and $a_1=2$.
\end{enumerate}

If $k=0$ and $a_1=3$, then the first row of the total picture has length $3$. Thus, in the second total row, to the right of the red line, we need three dots, making a total of $4$ dots. This is a valid solution, namely $(\alpha_1,\dots,\alpha_{n_2})=(2,5)$, $(\alpha_1+2, \alpha_2,\dots,\alpha_{n_2})=(4,5)$, $(A_1,\dots,A_{n_3})=(4)$ and $(a_1,\dots,a_{n_1})=(3,2,2,2)$. If we choose to continue and add $\alpha_3$, that means adding a new row completely to the right of the red line, which must be as long as the second total row, namely 4 dots. That again gives a valid solution $(\alpha_1,\dots,\alpha_{n_2})=(2,5,5)$, $(\alpha_1+2, \alpha_2,\dots,\alpha_{n_2})=(4,5,5)$, $(A_1,\dots,A_{n_3})=(4,5)$ and $(a_1,\dots,a_{n_1})=(3,2,2,3,2,2,2)$. We can continue this process and obtain the solution $(\alpha_1,\dots,\alpha_{n_2})=(2,(5)^{[l]})$ and $(a_1,\dots,a_{n_1})=((3,2,2)^{[l]},2,2)$ for all $l\geq 1$. Our legs are positively quasi-complementary, so $\alpha/p=[d_1,\dots,d_t]^-=[(5)^{[l]},4]^-$. Since $5-\frac{b}{a}=\frac{5a-b}{a}$, we have that \[
\begin{pmatrix}
\alpha\\
p
\end{pmatrix} = \begin{pmatrix}
5 & -1\\
1 & 0
\end{pmatrix}^{l} \begin{pmatrix}
4\\
1
\end{pmatrix}
\] for $l\geq 1$. This corresponds to family 6 in Theorem \ref{thm:torusknotlist}. We can compute $N=[0,3,2,2]^-=-\frac{3}{7}$. In other words, if $p_1=1$, $p_2=4$ and $p_{j+2}=5p_{j+1}-p_j$ for all $j\geq 0$ \cite[A004253]{OEIS}, we can say that \[ S^3_{p_jp_{j+1}-\frac{3}{7}}(T(p_j,p_{j+1})) \] bounds a rational homology ball for all $j\geq 1$. In this form it may not be obvious that the numerator of the surgery coefficient is a square, but in fact, $p_jp_{j+1}-\frac{3}{7}=\frac{V^2_{j+1}}{7}$ for $V_j$ being a sequence defined by $V_1=2$, $V_2=5$ and $V_{j+2}=5V_{j+1}-V_j$ for all $j\geq 0$ \cite[A003501]{OEIS}. It is a shifted so called Lucas sequence. The equality can be proven by first proving by induction that $p_{j+2}p_j-p_{j+1}^2=3$ for all $j\geq 0$, then noting that $V_{j+1}=p_{j+1}+p_j$ for all $j\geq 0$, and finally combining these equalities.

If $k=1$ and $a_1=2$ the argument goes the same way. The only way for the right and bottom legs to be quasi-complementary is if the Riemenschneider diagram to the right of the red line and the total diagram missing the bottom line coincide. By the same argument as above, it happens if and only if $(\alpha_1,\dots,\alpha_{n_2})=(3,(5)^{[l]})$ and $(a_1,\dots,a_{n_1})=(2,(3,2,2)^{[l]},2)$ for all $l\geq 0$. In this case $\alpha/p=[(5)^{[l]},5,2]^-$ and $N=[0,2,2,3]^-=-\frac{5}{7}$. This shows that if $Q_1=2$, $Q_2=9$ and $Q_{j+2}=5Q_{j+1}-Q_j$ for all $j \geq 1$, then \[ S^3_{Q_jQ_{j+1}-\frac{5}{7}}(T(Q_j,Q_{j+1})) \] bounds a rational homology ball for all $j\geq 1$. This corresponds to family 7 in Theorem \ref{thm:torusknotlist}. Just as before, we can show that \[ Q_jQ_{j+1}-\frac{5}{7} = \frac{(Q_j+Q_{j+1})^2}{7}. \]

\begin{figure}
\centering
\begin{tikzpicture}
	\begin{pgfonlayer}{nodelayer}
		\node [style={small_black_dot}, label={left:$-(2+k)$}] (0) at (-3, 5) {};
		\node [style={small_black_dot}, label={right:$-b_1$}] (1) at (-3, 3) {};
		\node [style={small_black_dot}, label={left:$-3$}] (2) at (-3, 1) {};
		\node [style={small_black_dot}, label={below:$-2-\beta_1$}] (3) at (-3, -7.75) {};
		\node [style={small_black_dot}, label={above:$-b_2$}] (11) at (-5, 3) {};
		\node [style=lookingempty] (12) at (-7, 3) {$\cdots$};
		\node [style={small_black_dot}, label={above:$-b_{m_1}$}] (13) at (-9, 3) {};
		\node [style={small_black_dot}, label={above:$-\beta_2$}] (14) at (-5, -7.75) {};
		\node [style=lookingempty] (15) at (-7, -7.75) {$\cdots$};
		\node [style={small_black_dot}, label={above:$-\beta_{m_2}$}] (16) at (-9, -7.75) {};
		\node [style={small_black_dot}, label={right:$-2$}] (18) at (-3, -1) {};
		\node [style=lookingempty] (19) at (-3, -3) {$\vdots$};
		\node [style={small_black_dot}, label={right:$-2$}] (20) at (-3, -5) {};
	\end{pgfonlayer}
	\begin{pgfonlayer}{edgelayer}
		\draw [style=new edge style 0] (0) to (1);
		\draw [style=new edge style 0] (1) to (2);
		\draw (13) to (12);
		\draw (12) to (11);
		\draw (11) to (1);
		\draw (14) to (3);
		\draw (15) to (14);
		\draw (16) to (15);
		\draw (2) to (18);
		\draw (18) to (19);
		\draw (19) to (20);
		\draw (20) to (3);
	\end{pgfonlayer}
\end{tikzpicture}
\caption{Choosing $n_1=n_2=1$ in Figure \ref{fig:liscagraph2234extendedgeneral} gives this star-shaped graph.}
\label{fig:aandalphagone}
\end{figure}

\begin{figure}
\centering
\includegraphics[width=.7\linewidth]{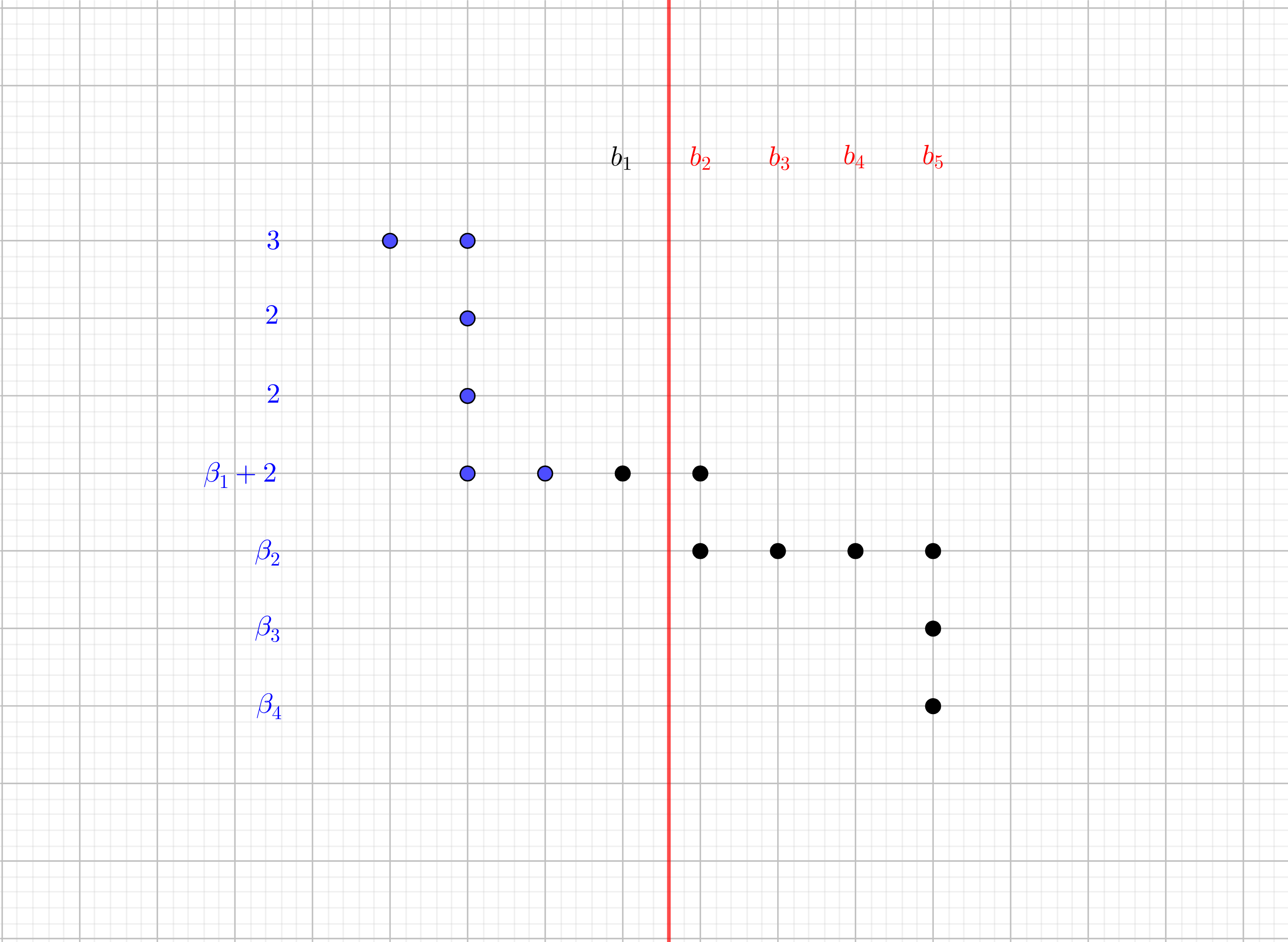}
\caption{}
\label{fig:makegrownlegsqc2}
\end{figure}

Returning to Figure \ref{fig:liscagraph2234extendedgeneral}, we can let the vertex of weight $-b_1$ be the only node. That forces $n_1=1$, so $(\alpha_1,\dots, \alpha_{n_2})=(2,\dots,2)$. Putting $a_1=2$ would give us complete freedom in choosing $(b_2,\dots, b_{m_1})$, so Proposition \ref{prop:aqcleg} applied to $Q/P=2+k$ gives that there are surgery coefficients $n$ such that $S^3_n(T(k+1,k+2))$, $S^3_n(T(k+2,(l+2)(k+2)-1))$ and $S^3_n(T(k+2,(l+1)(k+2)+1))$ bound rational homology $4$-balls. These families correspond to the entire families 1 and 2 in Theorem \ref{thm:torusknotlist}. (Note however, that a couple of subfamilies of these will also be realised if we choose $a_1>2$ because $(b_2,\dots,b_{m_1})=(2,\dots,2)$. These subfamilies have an especially ample supply of choices of surgery coefficents.) If $n_2>1$, then $m_2=1$ and $(b_1,\dots,b_{m_1})=(2,\dots,2)$. We will have three legs, namely $(-(2+k))$, $((-2)^{[\beta_1-2]})$ and $(-(1+a_1),(-2)^{[k]}, -(2+\beta_1), (-2)^{[a_1-2]})$. The first two can be quasi-complementary in two ways, but the generated pairs $(p,\alpha)$ are already known. The first and the third cannot be quasi-complimentary. The last two can also not be quasi-complementary if $n_2>1$. It is once again more interesting if $n_2=1$ and $m_2>1$ is allowed. We get the graph in Figure \ref{fig:aandalphagone}. The top and the bottom legs cannot be made quasi-complementary. The left and the bottom legs are the interesting case. Analogously to how we used Figure \ref{fig:makegrownlegsqc}, we can use Figure \ref{fig:makegrownlegsqc2} to show that $k=0$ and $(\beta_1,\dots, \beta_{m_2})=(4,(6)^{[l]})$ and $(b_1,\dots,b_{m_1})=(2,2,(3,2,2,2)^{[l]},2)$. This is in fact family (4) in \cite[Theorem 1.1]{GALL} and family 8 in Theorem \ref{thm:torusknotlist}.

Going back to Figure \ref{fig:liscagraph2234extendedgeneral}, we could also make the vertex of weight $-\alpha_1-\beta_1$ the only trivalent vertex, but that would require $m_1=n_1=1$ and thus all $\alpha_i$ and all $a_j$ are $2$'s. The top leg would not be able to be quasi-complementary to a sequence of $2$s, and the only way for the left and right legs to be quasi-complementary is if they are the legs $(-2)$ and $(-2,-2)$ in either order. This just gives us two new families of possible surgery coefficients on $T(2,3)$.

\printbibliography

\end{document}